\documentclass[11pt]{amsart}

\usepackage{amsmath, amsthm, amscd, amssymb, amsfonts, amsxtra}

\usepackage[varg]{txfonts}

\newtheorem{thm}{Theorem}[section]
\newtheorem{prop}[thm]{Proposition}
\newtheorem{cor}[thm]{Corollary}
\newtheorem{lem}[thm]{Lemma}

\theoremstyle{definition}
\newtheorem{defn}[thm]{Definition}

\theoremstyle{remark}
\newtheorem{rmk}[thm]{Remark}

\newcommand\al{\alpha}
\newcommand\bt{\beta}
\newcommand\G{\Gamma}
\newcommand\g{\gamma}
\newcommand\Dt{\Delta}
\newcommand\dt{\delta}
\newcommand\e{\varepsilon}
\newcommand\z{\zeta}
\renewcommand\th{\vartheta}
\renewcommand\k{\kappa}
\newcommand\Ld{\Lambda}
\newcommand\ld{\lambda}

\newcommand\x{\xi}

\newcommand\s{\sigma}
\newcommand\ph{\varphi}
\newcommand\ch{\chi}
\newcommand\ps{\psi}
\newcommand\Om{\Omega}
\newcommand\om{\omega}

\newcommand\Ccal{{\mathcal{C}}}
\newcommand\Scal{{\mathcal{S}}}

\newcommand\Ocal{\mathcal{O}}
\newcommand\Pcal{\mathcal{P}}

\newcommand\CC{{\mathbb{C}}}
\newcommand\NN{{\mathbb{N}}}
\newcommand\QQ{{\mathbb{Q}}}
\newcommand\RR{{\mathbb{R}}}
\newcommand\ZZ{{\mathbb{Z}}}

\newcommand\SL{\mathrm{SL}}

\newcommand\GL{\mathrm{GL}}
\newcommand\pl{\mathrm{Pl}}
\newcommand\npl{\widetilde{\pl}}
\newcommand\aid{\mathfrak{a}}
\newcommand\bid{\mathfrak{b}}
\newcommand\cid{\mathfrak{c}}
\newcommand\prm{\mathfrak{p}}
\newcommand\qrm{\mathfrak{q}}

\newcommand\ad{\mathbb{A}}
\newcommand\he{{\mathcal H}}
\newcommand\supp{\mathrm{supp\,}}
\newcommand\vol{\mathrm{Vol\,}}
\newcommand\sign{\mathrm{Sign\,}}
\newcommand\Gf{\Gamma} 
\newcommand\bc{C}
\newcommand\hg[1]{{}^{#1}\!\!}
\newcommand\B{{\mathrm B}}
\newcommand\K{{\mathrm K}}

\newcommand\isdef{\mathrel{:\mskip1mu=}}
\newcommand\divides{\mathrel{|}}
\newcommand\dividesnot{\mathrel{|\mskip-6mu/}}

\newcommand\re{\operatorname{Re}}

\newcommand\lbs{{\ld_0}}

\newcommand\ctt[1]{\mathrm{N}^{#1}}

\newcommand\ctr{\ctt{\k,r;\k',r'}}
\newcommand\nctt[1]{\mathrm{\tilde N}^{#1}}

\newcommand\nctr{\nctt{\k,r;\k',r'}}
\newcommand\ectr{\mathrm{Eis}^{\k,r;\k',r'}}
\newcommand\sg{\mathbf{s}}
\newcommand\lw[1]{_{\!#1}}
\newcommand\kls[1]{S\lw{#1}}
\newcommand\GmP[1]{\G_{\!\!P^{#1\vphantom{\k'}}}}
\newcommand\GmN[1]{\G_{\!\!N^{#1\vphantom{\k'}}}}
\newcommand\norm{\mathrm{N}}
\newcommand\nrm[1]{\norm\hspace{-.6pt}(#1)}
\newcommand\heop[1]{T(#1)}

\renewcommand\={\;=\;}
\newcommand\txtfrac[2]{{\textstyle\frac{#1}{#2}}}

\renewcommand\setminus{\smallsetminus}

\makeatletter
\newcommand\matc[4]{\left( {#1\@@atop #3}{#2\@@atop #4}\right)}
\newcommand\matr[4]{\left( {\hfill #1\@@atop\hfill #3}{\hfill
#2\@@atop\hfill #4}\right)}
\makeatother

\begin{document}
\title[Eigenvalues of Hecke operators] {Eigenvalues of Hecke operators
on Hilbert modular groups}
\author{Roelof W.Bruggeman, Roberto J.Miatello}

\begin{abstract}
Let $F$ be a totally real field, let $I$ be a nonzero ideal of the
ring of integers $\Ocal_F$ of~$F$, let $\G_0(I)$ be the congruence
subgroup of Hecke type of $G=\prod_{j=1}^d \SL_2(\RR)$ embedded
diagonally in $G$, and let $\ch$ be a character of $\G_0(I)$ of the
form $\ch\matc abcd = \ch(d)$, where $d\mapsto \ch(d)$ is a character
of $\Ocal_F$ modulo~$I$.

For a finite subset $P$ of prime ideals $\prm$ not dividing~$I$, we
consider the ring $\he^I$, generated by the Hecke operators
$\heop{\prm^2}$, $\prm \in P$ (see ~\S\ref{sssect-gho}) acting on
$(\G,\ch)$-automorphic forms on~$G$.

Given the cuspidal space $L^{2,\mathrm{cusp}}_\x\bigl(
\G_0(I)\backslash G,\ch\bigr)$, we let $V_\varpi$ run through an
orthogonal system of irreducible $G$-invariant subspaces
so that each $V_\varpi$ 
is invariant under~$\he^I$. For each $1\le j \le d$, let
$\ld_\varpi=(\ld_{\varpi,j})$ be the vector formed by the eigenvalues
of the Casimir operators of the $d$ factors of~$G$ on $V_\varpi$, and
for each $\prm\in P$, we take $\ld_{\varpi,\prm}\ge 0$ so that
$\ld_{\varpi,\prm}^2 -\nrm\prm$ is the eigenvalue on $V_\varpi$ of
the Hecke operator $\heop{\prm^2}$

For each family of expanding boxes $t\mapsto \Om_{t}$, as in
\eqref{Om-gen} 
in $\RR^d$, 
and fixed an interval $J_\prm$ in~$[0,\infty)$, for each $\prm\in P$,
we consider the counting function
 $$ N(\Om_{t};\,(J_\prm)_{\prm \in P}):=
 \sum_{\varpi,\, \ld_{\varpi}\in\, \Om_{t} \;:\;
\ld_{\varpi,\prm} \in J_\prm \;, \forall\prm\in P}
\left|c^r(\varpi)\right|^2.$$
Here $c^r(\varpi)$ denotes the normalized Fourier coefficient of
order~$r$ at~$\infty$ for the elements of $V_\varpi$, with
$r\in \Ocal_F' \smallsetminus \prm\,\Ocal_F'$ for every $\prm\in P$.

In the main result in this paper, Theorem~\ref{thm-afH}, we give,
under some mild conditions on the $\Om_{t}$, the asymptotic
distribution of the function $N(\Om_{t};\,(J_\prm)_{\prm \in P})$, as
$t\rightarrow \infty$. We show that at the finite places outside~$I$
the Hecke eigenvalues are equidistributed with respect to the
Sato-Tate measure, whereas at the archimedean places the eigenvalues
$\ld_\varpi$ are equidistributed with respect to the Plancherel
measure.

As a consequence,  if we fix an infinite place $l$ and
we prescribe
$\ld_{\varpi,j}\in \Om_j$
for all infinite places $j\ne l$ and$\ld_{\varpi,\prm} \in J_\prm $
for all finite places $\prm$ in~$P$
(for fixed intervals $\Om_j$ and $J_\prm$)
 and then
allow $|\lambda_{\varpi,l}|$ to grow to $\infty$, then
there are infinitely many such $\varpi$, and their positive density is
as described in Theorem~\ref{thm-afH}.
\end{abstract}

\subjclass{11F41 11F60 11F72 22E30}

\keywords{automorphic representations, Hecke operators, Hilbert
modular group, Plancherel measure, Sato-Tate measure}

\maketitle

\setcounter{tocdepth}1 \tableofcontents

\section{Introduction and discussion of main results}\label{sect-pdmr}
We work with a totally real number field $F$ of degree~$d$, the Lie
group $G=\SL_2(\RR)^d$ considered as the product of $\SL_2(F\lw j)$
for all archimedean completions $F\lw j\cong\RR$ of~$F$. The group
$\SL_2(F)$ is diagonally embedded in~$G$. We consider the congruence
subgroup $\G=\G_0(I)$ with $I$ a nonzero ideal in the ring of
integers $\Ocal_F=\Ocal$ of~$F$, a character $\ch$ of
$(\Ocal_F/I)^\ast$ inducing a character $\ch\matc abcd=\ch(d)$
of~$\G$, and a compatible central character determined by
$\x\in \{0,1\}^d$.

$L^2_\x(\G\backslash G,\ch)$ is the Hilbert space of classes of square
integrable functions transforming on the left by $\G$ according to
the character $\ch$, and transforming by the center $Z$ of $G$
according to the central character determined by~$\x$. We work with a
maximal orthogonal system $\{V_\varpi\}_\varpi$ of irreducible
subspaces in the cuspidal Hilbert subspace of
$L^2_\x(\G\backslash G,\ch)$. For each $\varpi$, there is an
eigenvalue vector $\ld_\varpi=(\ld_{\varpi,j})_j \in \RR^d$, where
$\ld_{\varpi,j}$ is the eigenvalue of the Casimir operator of the
factor $\SL_2(\RR)$ at place~$j$ in the product $G=\SL_2(\RR)^d$.

We normalize the Fourier terms of automorphic forms at the
cusp~$\infty$ as discussed in~\S\ref{sect-Four-coeff}. With this
normalization, one obtains Fourier coefficients $c^r(\varpi)$ that
are the same for all automorphic forms in~$V_\varpi$. The order $r$
of the Fourier terms runs through the inverse different $\Ocal_F'$
of~$\Ocal_F$.

In \cite{densII} (Theorem 4.5, Proposition 4.6 and Theorem 5.3) we
prove that, under some mild conditions on the family of compact sets
$t\mapsto \Om_t$ in $\RR^d$:
\begin{equation}\label{as1}
\sum_{\varpi,\, \ld_\varpi\in \Om_t} \left| c^r(\varpi)\right|^2 \=
\frac{2\sqrt{|D_F|}\,\vol(\G\backslash G)}{(2\pi)^d} \;\pl(\Om_t) \;
\bigl(1+o(1)\bigr)\qquad(t\rightarrow\infty)\,.
\end{equation}
(In \cite{densII} the factor $\vol(\G\backslash G)$ is not present. It
is due to a different normalization of measures discussed
in~\S\ref{sect-Four-coeff}.)
The factor $D_F$ is the discriminant of $F$ over~$\QQ$, the
$c^r(\varpi)$ are Fourier coefficients, and $\pl$ denotes the
Plancherel measure given by $\pl = \otimes_j \pl_{\x_j}$ on $\RR^d$,
where
\begin{align}\label{pl-def}
\pl_0(f) &\= \int_{1/4}^\infty f(\ld)
\tanh\pi\sqrt{\ld-\txtfrac14}\; d\ld
\displaybreak[0]\\
\nonumber
&\qquad\hbox{} + \sum_{b\geq 2,\, b\equiv 0 \bmod2} (b-1)\; f\left(
\txtfrac b2\left(1-\txtfrac b2\right) \right),
\displaybreak[0]
\\
\nonumber
\pl_1(f)&\= \int_{1/4}^\infty f(\ld)
\coth\pi\sqrt{\ld-\txtfrac14}\; d\ld
\displaybreak[0]\\
\nonumber
&\qquad\hbox{} + \sum_{b\geq 3,\, b\equiv 1 \bmod2} (b-1)\; f\left(
\txtfrac b2\left(1-\txtfrac b2\right) \right)
\end{align}

In this paper we work with a family of the following type:
\begin{equation}\label{Om-gen}
\Om_t \= [-t,t]^Q \times \prod_{j\in E} [A_j,B_j]\,,
\end{equation}
where $\{1,\ldots, d\}= Q \sqcup E$ is a partition of the archimedean
places of~$F$ for which $Q\neq \emptyset$. The end points $A_j$ and
$B_j$ of the fixed interval with $j\in E$ are not allowed to lie in
the set
$\bigl\{ \frac b2(1-\frac b2)\;:\; b \equiv \x_j\,,\; b>1\bigr\}$, of
discrete series eigenvalues. The variable $t$ tends to infinity. The
set $E$ can be empty, but the set $Q$ has to contain at least one
place.
\smallskip

The asymptotic formula \eqref{as1} shows that if the box
$\prod_{j\in E} [A_j,B_j]$ has positive Planche\-rel measure
$\otimes_{j\neq l} \pl_{\x_j}$, then there are infinitely many
eigenvalue vectors $\ld_\varpi$ that project to this box.

If one of the factors $[A_j,B_j]$ has zero density for $\pl_{\x_j}$,
which happens in particular if
$[A_j,B_j] \subset \bigl(0,\frac 14\bigr)$, then the asymptotic
formula has no meaning. For such a situation it is better to use the
following formulation:
\begin{equation}\label{as}
\sum_{\varpi,\, \ld_\varpi\in \Om_t} \left| c^r(\varpi)\right|^2 \=
\frac{2\sqrt{|D_F|}\,\vol(\G\backslash G)}{(2\pi)^d} \pl(\Om_t)
+o\bigl(V\lw1(\Om_t)\bigr)
\qquad(t\rightarrow\infty)\,,
\end{equation}
where the reference measure $V\lw1$ has product structure
$V\lw1= \otimes_j V\lw{1,\x_j}$ with
\begin{align}\label{V-def}
\int h\, dV\lw{1,0} &\;=\; \frac12 \int_{5/4}^\infty h(\ld)\,d\ld
+ \frac12 \int_0^{5/4} |\ld-1/4|^{-1/2}\, d\ld \\
\nonumber
&\qquad\hbox{}+ \sum_{\bt>0\,,\; \bt\equiv\frac12(1)} \bt\,
h(1/4-\bt^2)\,,\displaybreak[0]\\
\nonumber
\int h\, dV\lw{1,1} &\;=\; \frac12 \int_{5/4}^\infty h(\ld)\,d\ld
+ \frac12 \int_{1/4}^{5/4} |\ld-1/4|^{-1/2}\, d\ld \\
\nonumber
&\qquad\hbox{}+ \sum_{\bt>0\,,\; \bt\equiv0(1)} \bt \,h(1/4-\bt^2)\,.
\end{align}
This measure $V\lw1$ is positive on all sets in which eigenvalue
vectors can occur, and is comparable to $\pl$ near points $\ld$ for
which all coordinates stay a positive distance away from
$\bigl(0,\frac 14\bigr)$. So the asymptotic formula does not exclude
exceptional $\ld_{\varpi,j}$, but only limits their density.

In this paper we consider Hecke operators of the form $\heop{\prm^2}$
for primes $\prm$ not dividing~$I$. If the prime ideal
$\prm=\pi_\prm \Ocal$ is principal, then the action of the Hecke
operator $\heop{\prm^2}$ on $(\G,\ch)$-automorphic functions $f$
on~$G$
(i.e., transforming on the left according to the character $\ch$
of~$\G$) is given by
\begin{align}\label{Tp2}
f|\heop{\prm^2}(g) &\= \ch(\pi_\prm) f\left(
\matc{\pi_\prm}00{1/\pi_\prm} g \right)
+ \sum_{b\in \Ocal/\prm} f\left( \matc1{b/\pi_\prm}01 g\right)\\
\nonumber
&\qquad\hbox{} + \sum_{b\in \Ocal/\prm^2} \ch(\pi_\prm)^{-1} f \left(
\matc {1/\pi_\prm}{b/\pi_\prm}0{\pi_\prm}g\right)\,.
\end{align}
This does not depend on the choice of the generator~$\pi_\prm$. The
Hecke operator in \eqref{Tp2} preserves $(\G,\ch)$-automorphy, square
integrability and cuspidality. For non-principal $\prm \nmid I$, we
will show that there are also Hecke operators $\heop{\prm^2}$ with
similar properties (see \S\ref{sssect-gho}). They generate a
commutative algebra of symmetric bounded operators on
$L^{2,\mathrm{cusp}}_\x(\G\backslash G,\ch)$. Furthermore, the
orthogonal system $\{V_\varpi\}$ can be chosen in such a way that
each operator acts on $V_\varpi$ by multiplication by a fixed scalar,
One can show that the eigenvalue of $\heop{\prm^2}$ on $V_\varpi$ is
a real number with absolute value at most $\nrm\prm^2+\nrm\prm+1$.
(See \eqref{lbd}.)
Hence there is
\begin{equation}\label{ldprm}
\ld_{\varpi,\prm} \in \bigl[ 0,\sqrt{\nrm\prm^2+1}\bigr]
\end{equation}
such that $\ld_{\varpi,\prm}^2-\nrm\prm$ is the eigenvalue of
$\heop{\prm^2}$ in~$V_\varpi$ (see \eqref{ldpbd}). We use this
parametrization of Hecke eigenvalues since we do not wish to depart
much from the usual parametrization by Satake parameters (see
\S\ref{sect-lha}).

We note that if $\chi =1$ and if the prime $\prm \nmid I$
is of the form $\prm = \Ocal\pi_\prm$ with $\pi_\prm$ totally
positive, then one can define a Hecke operator $\heop\prm$ by the
formula
\begin{align}
(f|\heop{\prm}(g) \= f \left( \matc{\sqrt{\pi_\prm}}00{\tfrac
1{\sqrt{\pi_\prm}}} g\right)
+ \sum_{b\in \Ocal/ \prm} f \left( \matc{\frac
1{\sqrt{\pi_\prm}}
}{\frac b{\sqrt{\pi_\prm}}} 0 {\sqrt{\pi_\prm}} g \right)\,.
 \end{align}
 and $\heop\prm$ satisfies $\heop\prm^2$ = $\heop{\prm^2} + \nrm\prm$.
 In such a situation we can further arrange the system $\{V_\varpi\}$
 to be such that $\heop\prm$ has eigenvalue $\ld_{\varpi,\prm}$ or
 $-\ld_{\varpi,\prm}$ on $V_\varpi$, with $\ld_{\varpi,\prm}\ge 0$ as
 chosen above. In general, one cannot define $\heop\prm$ for all prime
 ideals~$\prm$.
 \smallskip

We now define for each prime ideal $\prm \nmid I$ 
the measure $\Phi_\prm$ on $\RR$ by
\begin{equation}\label{Phi-def}
\Phi_\prm (f) \= \frac1{\pi\; \nrm\prm}\int_0^{2\sqrt{\nrm\prm}}
f(\ld)\, \sqrt{4\;\nrm\prm - \ld^2}\, d\ld\,.
\end{equation}
By writing this as
\begin{equation}\label{ST}
\Phi_\prm(f) \= \frac 2\pi \int_0^{\pi/2} f\bigl(2\sqrt{\nrm\prm}\,
\cos\th\bigr)
\sin^2\th\, d\th
\end{equation}
we see that $\Phi_\prm$ can be viewed as the Sato-Tate measure on even
functions~$f$.
\smallskip

The main goal of this paper is to prove the following distribution
result:
\begin{thm}\label{thm-afH} Let $t\mapsto \Om_t$ be a family of compact
sets in $\RR^d$ as in \eqref{Om-gen} and let $P$ be a finite set of
prime ideals in $\prm \in \Ocal_F$, $\prm \nmid I$. For each
$\prm\in P$ let $J_\prm$ be an interval in~$[0,\infty)$. Then for any
$r\in \Ocal_F'$ such that $r\not\in \prm\,\Ocal_F'$ for every
$\prm\in P$, we have:
\begin{align}\label{afH}
&\sum_{\varpi,\, \ld_\varpi\in \Om_t \;:\; \ld_{\varpi,\prm} \in
J_\prm \;, \forall\prm\in P} \left|c^r(\varpi)\right|^2 \\
\nonumber
&\qquad\= \frac{2\sqrt{|D_F|}\,\vol(\G\backslash G)}{(2\pi)^d} \,
\biggl( \pl(\Om_t)
\prod_{\prm\in P} \Phi_\prm(J_\prm) + o\bigl(V\lw1(\Om_t)\bigr)
\biggr)\,.
\end{align}
\end{thm}
This shows that in the Hilbert modular case Hecke eigenvalues are
equidistributed with respect to the Sato-Tate measure. See
p.~106--107 of~\cite{Ta}. In the case $F=\QQ$ and $I=\ZZ$ the result
in Theorem~\ref{thm-afH} is a consequence of Proposition 4.10
in~\cite{Br78}.
(In this case, since $d=1$, the set $E$ has to be empty.)

The theorem stays true for more general families $t\mapsto \Om_t$, as
discussed in~\S\ref{sect-cond} in the appendix.

\begin{cor}Let $E$ be a set of archimedean places, with $0\leq\#E<d$,
and let $S$ be a finite set of finite places outside~$I$. Suppose
that $I_v\subset\RR$ is a bounded interval for each $v\in E\cup S$.
Suppose that the $I_j$ with $j\in E$ satisfy the condition on the end
points mentioned below~\eqref{Om-gen}. If 
\[\prod_{j\in E} \pl_{\x_j}(J_j)\; \prod_{\prm\in S, \text{
 finite}}\Phi_\prm(J_\prm)>0 \]
then there are infinitely many $\varpi$ such that
$\ld_{\varpi,v}\in J_v$ for all~$v\in S$.
\end{cor}

The \emph{Ra\-ma\-nujan-Petersson} conjecture states that the
eigenvalues should satisfy
\begin{equation}\label{RP}
 \ld_{\varpi,j} \in (-\infty,0] \cup[1/4,\infty),\; \forall\,
 j\;,\quad \ld_{\varpi,\prm} \in [0,\sqrt{2\;\nrm\prm}],\; \forall\,
 {\prm\nmid I}\;.
\end{equation}
In the language of representation theory this means that local factors
$\varpi_v$ of $\varpi$ cannot be in the complementary series for any
place $v$ of~$F$ outside~$I$. Theorem \ref{thm-afH} gives support to
this conjecture: If $v=j$ is an archimedean place, then we take
$E=\{j\}$ and $[A_j,B_j]\subset \bigl(0,\frac14\bigr)$. If $v=\prm$
is a finite place not dividing~$I$, we take $S=\{\prm\}$ and
$J_\prm\subset\bigl(2\sqrt{\nrm\prm},\infty\bigr)$. We conclude from
the theorem that exceptional eigenvalues such that
$\ld_{\varpi,j} \in [A_j,B_j]$ or $\ld_{\varpi,\prm}\in J_\prm$ are
relatively scarce: the sum in the left hand side of~\eqref{afH} is
$o\bigl(V_1(\Om_t) \bigr)$ as $t\rightarrow\infty$. The theorem gives
more: the $\ld_{\varpi,v}$ are distributed according to the
Plancherel measure if $v$ is an archimedean place, and according to
the Sato-Tate distribution if $v$ is a finite place outside~$I$.
\medskip

The paper is organized as follows. In \S\ref{sect-prelims} we
introduce some notations and recall some facts on automorphic forms.
To handle Hecke operators it is convenient to work in an adelic
context recalled in~\S\ref{sect-afag}. In \S\ref{sect-Ha} we discuss
the Hecke algebra $\he^I$ and the relation of its eigenvalues to
Fourier coefficients. In \S\ref{sect-ad} we give the proof of the
main result, Theorem~\ref{thm-afH}. To this end, we need to
generalize the asymptotic result \eqref{as} to a sum in which
$|c^r(\varpi)|^2$ is replaced by a product of Fourier coefficients of
possibly different order and at possibly different cusps. This
generalized asymptotic result is Theorem~\ref{thm-af'}, proved by
using the sum formula given in Theorem~\ref{thm-sf}. We have included
this auxiliary material in two Appendices to avoid interrupting the
flow of proof of the main result.

\section{Preliminaries} \label{sect-prelims}
In this section we will introduce some notations and recall known
facts on automorphic forms and Hecke operators in our context.

\subsection{Notations}\label{sect-not} For any ring $R$ we denote:
\begin{equation} n(x) \= \matc 1x01 \text{ for }x\in R\,,\qquad h(t)\=
\matc t00{t^{-1}}\text{ for }t\in R^\ast\,.\\
\end{equation}
In the case $R=\RR^d$ we also use
\begin{align} a(y)&\= \left( \matc {\sqrt
{y_1}}00{1/\!\sqrt{y_1}},\ldots,\matc {\sqrt
{y_d}}00{1/\!\sqrt{y_d}}\right) \qquad \text{ for }y\in
(0,\infty)^d\\\nonumber
\\
\nonumber
k(\th)&\= \left( \matr{\cos\th_1}{\sin\th_1}{-\sin
\th_1}{\cos\th_1},\ldots, \matr{\cos\th_d}{\sin\th_d}{-\sin
\th_d}{\cos\th_d}\right) \qquad \text{ for }\th\in \RR^d
\end{align}
and write
\begin{align}N&\=\bigl\{ n(x)\;:\; x\in \RR^d\bigr\}\,,\qquad A\=
\bigl\{ a(y)\;:\; y\in
(0,\infty)^d\bigr\}\,,\\
\nonumber
K&\=\bigl\{ k(\th)\;:\; \th\in \RR^d\bigr\}\,.\end{align}
The center $Z$ of~$G$ is the set of $h(\z)$ with
$\z\in \{1,-1\}^d\subset\RR^d$.
\smallskip

The function $S:\RR^d\rightarrow\RR$ given by $S(x)=\sum_j x_j$
extends $\mathrm{Tr}_{F/\QQ}:F\rightarrow\QQ$.

\subsection{Cusps} As in \S2.1.4 of \cite{BM9} we choose a set $\Pcal$
of representatives of the $\G$-equivalence classes of cusps. We take
$\infty$ as representative of $\G\,\infty$. For each $\k\in \Pcal$ we
choose $g_\k\in \SL_2(F)$ such that $\k=g_\k\,\infty$. For
$\k=\infty$ we can and do take $g_\infty=1$. We use the notations
$N^\k=g_\k N g_\k^{-1}$ and $P^\k=g_\k NAZ g_\k^{-1}$, and we put
$\GmP \k \isdef \G \cap P^\k$ and $ \GmN\k \isdef \G \cap N^\k$. So
$\GmP\k$ equals the subgroup of $\G$ fixing~$\k$, and has $\GmN\k $
as a normal commutative subgroup of the form
$\GmN\k = \bigl\{ g_\k n(\x)g_\k^{-1}\; \x\in M_\k\bigr\}$ for some
fractional ideal $M_\k$ of~$F$, which is a lattice in $\RR^d$, under
the embedding of $F$ in $\RR^d$. There is the dual lattice $M_\k'$
consisting of all $r\in \RR^d$ such that $S(rx)\in \ZZ$ for all
$x\in M_\k$. So $M_\k'$ is the fractional ideal $\Ocal'\, M_\k^{-1}$,
where $\Ocal'$ is the inverse of the different ideal. Note that
$M_\k$ and $M_\k'$ depend on the choice of~$g_\k$, but $\GmP\k$ and
$\GmN\k$ do not. For the cusp $\infty$ we have taken $g_\infty=1$, so
$M_\infty=\Ocal$ and $M_\infty'=\Ocal'$.

$\Pcal\lw\ch$ is the subset of $\k\in \Pcal$ for which the fixed
character $\ch$ of $\G$ is trivial on $\GmN\k$. For
$\k\in \Pcal\lw\ch$, the dual lattice $M_\k'$ describes the
characters of $g_\k N g_\k^{-1}$ that are trivial on~$\GmN\k$. For
general $\k\in \Pcal$ we put
\begin{equation}
\label{Mtk-def}
\tilde M_\k' \isdef \bigl\{ r \in \RR^d\;:\; \ch(g_\k n(x)
g_\k^{-1} ) \= e^{2\pi i S(r x)}\text{ for every }x\in M_\k\bigr\}\,.
\end{equation}
Thus, $\tilde M_\k'$ is a shift of $M_\k'$ inside~$\RR^d$, and
$\k\in \Pcal_\ch$ if and only if $0\in \tilde M_\k'$.

\subsection{Automorphic functions and Fourier terms}
By a $(\G,\ch)$-au\-to\-mor\-phic function on~$G$ we mean any function
on~$G$ that satisfies $f(\g g) = \ch(\g) f(g)$ for all $\g\in \G$ and
$g\in G$. We call a $(\G,\ch)$-automorphic function an
\emph{automorphic form} if $f$ is an eigenfunction of the $d$ Casimir
operators of the factors of $G=\SL_2(\RR)^d$, and if it has a
\emph{weight} $q\in \ZZ^d$, {\sl i.e.},
$f(gk(\th)) = f(g)e^{iS(q\th)}$ for any $\th \in \RR^d$, where
$S(q,\th)=\sum_j q_j\th_j$. All automorphic forms on~$G$ are
real-analytic functions.

Let $f$ be a continuous $(\G,\ch)$-automorphic function on~$G$. For
each $\k\in \Pcal$ the function $x\mapsto f\bigl( g_\k n(x) g )$ on
$\RR^d$ transforms according to a character
$\x\mapsto e^{2\pi i S(r\x)}$ of $M_\k$ for some $r\in \tilde M_\k'$.
So it has an absolutely convergent Fourier expansion
\begin{align}
f(g_\k g) &\= \sum_{r\in \tilde M_\k'}
F_{\k,r}f(g)\,,\displaybreak[0]\\
\nonumber
F_{\k,r}(g) &\= \frac1{\vol( \RR^d/M_\k)} \int_{\RR^d/ M_\k} e^{-2\pi
i S(rx)} f(g_\k n(x)
g)\, dx\,.
\end{align}

In \cite{BM9} and \cite{densII} we considered only the Fourier
expansion at the cusp~$\infty$. For the relation with Hecke operators
we need to consider in this paper Fourier expansions at other cusps
as well.

We call an $(\G,\ch)$-automorphic function \emph{cuspidal} is all its
Fourier terms of order zero vanish. Sin $0\in \tilde M_\k'$ only if
$\k\in \Pcal_\ch$, the function $f$ is cuspidal if $F_{\k,0}f=0$ for
all $\k\in \Pcal_\ch$.

\subsection{Automorphic representations}\label{sect-autrepr}
Let $L^2(\G_0(I)\backslash G,\ch)$ denote the Hilbert space of classes
of functions transforming according to $\ch$ i.e.
$f(\g g) = \ch(\g)f(g)$ for any $\g\in \G_0(I)$ and $g \in G$. The
group $G$ acts unitarily on this Hilbert space by right translation.
This space is split up according to central characters, indicated by
$\x\in \{0,1\}^d$. By $L^2_\x(\G_0(I)\backslash G,\ch)$ we mean the
subspace on which the center acts by
\[ \left(\matr{\z_1}00{\z_1},\ldots,\matc{\z_d}00{\z_d} \right)
\mapsto \prod_j \z_j^{\x_j},\] where $\z_j\in \{1,-1\}$. This subspace
can be non-zero only if the following compatibility condition holds:
\begin{equation}\label{ch-x-comp}
\ch(-1) = \prod_j (-1)^{\x_j}.
\end{equation}
We assume this throughout this paper.

There is an orthogonal decomposition
\begin{equation}\label{L2decomp}
L^2_\x(\G_0(I)\backslash G,\ch) =
L^{2,\mathrm{cont}}_\x(\G_0(I)\backslash G,\ch) \oplus
L^{2,\mathrm{discr}}_\x(\G_0(I)\backslash G,\ch).
\end{equation}
The $G$-invariant subspace
$L^{2,\mathrm{cont}}_\x(\G_0(I)\backslash G,\ch)$ can be described by
integrals of Eisenstein series and the orthogonal complement
$L^{2,\mathrm{discr}}_\x(\G_0(I)\backslash G,\ch)$ is a direct sum of
closed irreducible $G$-invariant subspaces. If $\ch=1$, the constant
functions form an invariant subspace. All other irreducible invariant
subspaces have infinite dimension. They are cuspidal and span the
space $L_\x^{2,\mathrm {cusp}}(\G_0(I)\backslash G,\ch)$, the
orthogonal complement of the constant functions in
$L^{2,\mathrm{discr}}_\x(\G_0(I)\backslash G,\ch)$.

We fix a maximal orthogonal system $\left\{ V_\varpi \right\}_\varpi$
of irreducible subspaces in the Hilbert space
$L^{2,\mathrm{cusp}}_\x(\G_0(I)\backslash G,\ch)$. This system is
 unique if all $\varpi$ are inequivalent. In general, there might be
multiplicities, due to oldforms.
\medskip

Each irreducible automorphic representation $\varpi$ of $G=
\prod_j \SL_2(\RR)$ is the tensor product $\otimes_j \varpi_j$ of
irreducible representations of $\SL_2(\RR)$. Here and in the sequel,
$j$ is supposed to run over the $d$ archimedean places of~$F$.

The factor $\varpi_j$ can (almost) be characterized by the eigenvalue
$\ld_{\varpi,j}$ of the Ca\-si\-mir operator of $\SL_2(\RR)$, and the
central character, which is indicated by $\x_j$.

The eigenvalues $\ld_{\varpi,j}$ are known to be in the following
subsets of~$\RR$ depending on the central character~$\x_j$:
\begin{equation}
\begin{aligned}
 \bigl\{ \txtfrac b2-\txtfrac {b^2}4 \;:\; b\geq 2\text{ even} \bigr\}
&\;\cup\; \bigl[ \lbs,\infty\bigr)&\quad&\text{if }\x_j=0\,,\\
 \bigl\{ \txtfrac b2-\txtfrac {b^2}4 \;:\; b\geq 3\text{ odd} \bigr\}
&\;\cup\; \bigl[\txtfrac14,\infty\bigr)& &\text{if }\x_j=1\,,
\end{aligned}
\end{equation}
where $\lbs\in\bigl( 0,\frac14\bigr]$. By a conjecture due to Selberg
in the spherical case, it is expected that one can take
$\lbs=\frac14$. The best result at present is
$\frac14 \geq \lbs \geq \frac14 - \bigl(\frac 19\bigr)^2$, due to
Kim-Shahidi and Kim-Shahidi-Sarnak. See \cite{KS1}. We call
$\ld_{\varpi} = \left( \ld_{\varpi,j} \right) \in \RR^d$ the
\emph{eigenvalue vector} of the representation $\varpi$. As discussed
in \S1.1 of~\cite{densII}, spectral theory shows that the set of
eigenvalue vectors $\{\ld_\varpi\}$ is discrete in $\RR^d$.

The correspondence between values of $\ld=\ld_{\varpi,j}$ and
equivalence classes of unitary representations of $\SL_2(\RR)$ of
infinite dimension is one-to-one if $\ld >0$ for $\x=\x_j=0$, and if
$\ld>\frac14$ if $\x=1$. In the other cases, $\ld=\frac
b2-\frac{b^2}4$ with $b\in \ZZ_{\geq 1}$, $b\equiv \x\bmod 2$. In
these cases, there are two equivalence classes, one with lowest
weight $b$ (holomorphic type), and one with highest weight $-b$
(antiholomorphic type). If $b\geq 2$, representations of these classes
occur discretely in $L^2\bigl(\SL_2(\RR)\bigr)$, and are called
\emph{discrete series representations}. The representations in the
case $b=1$ are sometimes called \emph{mock discrete series}. They do
not occur discretely in $L^2\bigl( \SL_2(\RR) \bigr)$. All these
representations, discrete series or not, may occur as an irreducible
component of $L^{2,\mathrm{cusp}}_\x\left( \G_0(I)\backslash
G,\ch\right)$.

\subsection{Fourier coefficients}\label{sect-Four-coeff}
In the version in~\cite{densII} of the asymptotic formula~\eqref{as1}
the factor $\vol(\G\backslash G)$ is not present, due to a different
choice of the norm in the Hilbert space $L^2_\x(\G\backslash G,\ch)$.
There we used $\|f\|^2 = \int_{\G\backslash G}|f|^2\, dg$, with the
Haar measure $dg$ indicated in \S2.1 of~\cite{BM9}. Here we prefer to
use the norm
\begin{equation} \label{newnorm}
\|f\| = \Bigl( \frac 1{\vol(\G\backslash G)} \int_{\G\backslash G}
|f|^2\, dg\Bigr)^{1/2}\,.
\end{equation}
With this choice the norm $\|f\|$ of $f\in L^2_\x(\G\backslash G,\ch)$
does not change if we consider $f$ as an element of
$L^2_\x(\G_{\!1}\backslash G,\ch)$ for a subgroup $\G_{\!1}$ of
finite index in~$\G$.

As discussed in \S2.3.4 in \cite{BM9}, the Fourier expansion of one
automorphic form in $V_\varpi$ determines the Fourier expansion of
any automorphic form in $V_\varpi$. This is valid at all cusps, not
only at the cusp~$\infty$ considered in~\cite{BM9} and~\cite{densII}.
We review the choice of the Fourier coefficients of an automorphic
representation~$V_\varpi$, indicating the differences in comparison
with~\cite{BM9}. We put references to formulas and sections
in~\cite{BM9} in italics.

Like in {\it(2.27)}, we determine the Fourier coefficients
$c^{\k,r}(\varpi)$, for $\k\in \Pcal$ and $r\in \tilde M_\k'$, by
\begin{equation}\label{Fourc-def}
F_{\k,r} \ps_{\varpi,q} \= c^{\k,r}(\varpi)\, d^r(q,\nu_\varpi)\,
W_q(r,\nu_\varpi)\,.
\end{equation}
Here $d^r(q,\nu_\varpi)$ is the factor with exponentials and gamma
functions in {\it (2.28)}, with a choice of
$\nu_\varpi = (\nu_{\varpi,j})_{1\leq j \leq d}\in $ such that
$\ld_{\varpi,j}=\frac14-\nu_{\varpi,j}^2$. The ``Whittaker function''
$W_q(r,\nu_\varpi)$ on~$G$ is a standard function on~$G$, given in
{\it (2.12)}, with weight $q$ and the right transformation behavior
for translations by elements of~$N$ on the left. The weight~$q$ in
\eqref{Fourc-def} runs over all weights that occur in~$V_\varpi$. The
$\ps_{\varpi,q}$ form a basis of the space of $K$-finite vectors
in~$V_\varpi$. They are related by the action of the Lie algebra of
$G$ and have norms as indicated in~{\it(2.26)}. Then this determines
the Fourier coefficients $c^{\k,r}(\varpi)$, independent of the
weight~$q$. The difference with \cite{BM9} is the choice of the norm
in \eqref{newnorm}, which can be summarized as follows:
\[
\begin{aligned}
&f\in L^2_\x(\G\backslash G,\ch):& \|f\|_{\mathrm{here}}&\=
\frac1{\sqrt{\vol(\G\backslash G)}}\, \|f\|_{\mathrm{there}}\,,\\
&\ps_{\varpi,q}\text{ in (2.26) of \cite{BM9}:}&
\ps_{\varpi,q}^{\mathrm{here}}&\= \sqrt{\vol(\G\backslash G)}\,
\ps_{\varpi,q}^{\mathrm{there}}\,,\\
&c^r(\varpi)\text{ in (2.27) of \cite{BM9}:}&\quad
c^r_{\mathrm{here}}(\varpi) &\= \sqrt{\vol(\G\backslash G)}\,
c^r_{\mathrm{there}}(\varpi)\,.
\end{aligned}
\]
As a consequence, we use now Fourier coefficients
$c^r(\varpi)=c^{\infty,r}(\varpi)$ that are equal to
$\sqrt{\vol(\G\backslash G)}$ times the corresponding coefficients
in~\cite{BM9} and~\cite{densII}. This causes the factor
$\vol(\G\backslash G)$ in \eqref{as1}, which stayed hidden in the
normalization used in~\cite{densII}. The wish to make the dependence
on the ideal $I$ in $\G=\G_0(I)$ explicit is our motivation to go
over to the norm in \eqref{newnorm}.

In the choice of the $c^{\k,r}(\varpi)$ there is for each $\varpi$ the
freedom of a complex factor with absolute value one. In the result we
work with absolute values, so these choices do not influence the
results of this paper.

\subsection{Automorphic forms on the adele group} \label{sect-afag}
 We now consider the embedding of $\G=\G_0(I) \subset \SL_2(F)$ in
$\SL_2(\ad_f)$, where $\ad_f$ is the group of finite $F$-adeles.

We recall that $\ad_f $ is the subset of $\prod_\prm F\lw\prm$ of
those $a=(a_\prm)_\prm \in \prod_\prm F\lw\prm$ such that $a_\prm
\in \Ocal_\prm$ for all but finitely many primes $\prm$ in~$\Ocal$.
By $\Ocal_\prm$ we denote the closure of $\Ocal$ in the completion
$F\lw\prm$ of $F$ at place~$\prm$. The ring $\ad_f$ contains the
subring $\bar\Ocal = \prod_\prm \Ocal_\prm$, which is the completion
of $\Ocal$ for the topology in which the system of neighborhoods
of~$0$ is generated by the non-zero ideals in~$\Ocal$.

The local ring $\Ocal_\prm$ has maximal ideal $\bar\prm$, which is the
closure of $\prm\subset\Ocal$ in $\Ocal_\prm$. The closure $\bar
I_\prm$ of~$I$ in $\Ocal_\prm$ is equal to $\Ocal_\prm$ if $\prm\nmid
I$, and is a nonzero ideal contained in $\bar\prm$ if $\prm\divides
I$.

For each $c\in \SL_2(F)$ the group $\G \cap c\G c^{-1}$ has finite
index in~$\G=\G_0(I)$. This system of subgroups is cofinal with the
system of all congruence subgroups of~$\G$. The closure $\bar\G$ of
 $\G $ in $\SL_2(\ad_f)$ is the completion of~$\G$ in the topology for
which the system of neighborhoods of the unit element is generated by
the congruence subgroups. There is a decomposition as a direct
product
$$\bar\G =\prod_\prm
\bar\G_\prm, \quad \textrm{ where } \bar\G_\prm = \left\{
\begin{array}{l} \SL_2(\Ocal_\prm)\, \textrm{ if } \prm\nmid I, \\
  \left\{ \matc abcd\in \SL_2(\Ocal_\prm)\;:\; c\in \bar
I_\prm\right\}\, \textrm{ if } \prm\divides I.
\end{array}\right.
$$
The isomorphism $(\Ocal/I)^\ast \cong \prod_{\prm\divides I} (
\Ocal/\prm^{v_\prm(I)})^\ast$ implies that the character $\ch$ mod
$I$ can be written uniquely as a product
$\ch (x) = \prod_{\prm\divides I} \ch_\prm(x)$ for $x\in \Ocal$
relatively prime to~$I$, where $\ch_\prm$ is a character of
$(\Ocal/\prm^{v_\prm(I)})^\ast \cong (\bar \Ocal_\prm/\bar
I_\prm)^\ast$. In particular, $\ch_\prm$ induces a character of
$\Ocal_\prm^\ast$.

We now define a character $\hat \ch$ on the subgroup $$\{u \in
\ad_f^\ast :u_\prm\in \Ocal_\prm^\ast, \textrm{ if } \prm\divides I,
\textrm{ and } u_\prm\in F\lw\prm^\ast \textrm{ if } \prm\nmid I
\},$$ by setting $\hat \ch(u) = \prod_{\prm\divides I}
\ch_\prm(u_\prm)$.
If $x\in \Ocal$ is relatively prime to $I$ then, for the diagonal
embedding $\Ocal \subset\ad_f^\ast$, we have $\ch(x)
= \hat\ch(x)$. Thus we have extended the character $\ch$ of
$(\Ocal/I)^\ast$.

Recall that we use the symbol $\ch$ also to denote the character
$\ch\matc abcd = \ch(d)$ of $\G$. We may extend this character $\ch$
to the subgroup $ \{ g\in \SL_2(\ad_f) : g_\prm \in
\bar\G_\prm, \textrm{ for } \prm\divides I\}$ of $\SL_2(\ad_f)$ by
\begin{equation}
\hat \ch \matc abcd \= \prod_{\prm\divides I} \ch_\prm(d) \= \hat
\ch(d)\,.
\end{equation}
This character is trivial on $\SL_2(F\lw\prm)$ for all $\prm\nmid I$.
For all $\prm\divides I$, we denote by ${\hat \ch}_\prm$ the
restriction of $\hat \ch$ to~$\bar\G_\prm$.
\smallskip

The adele ring $\ad$ of $F$ is the direct product $\ad_\infty
\times\ad_f$, with $\ad_\infty=\RR^d$. So $G=\SL_2({\RR })^d =
\SL_2(\ad_\infty)$. The group $\SL_2(F)$ can be viewed as embedded
discretely in~$\SL_2(\ad)$. Any $(\G,\ch)$-automorphic function $f$
on~$G$ determines
uniquely a corresponding function $f_a$ on $\SL_2(\ad)$ by
\begin{equation}
f_a(c(g_\infty,u)) \= \hat\ch(u)^{-1}f(g_\infty)\qquad g_\infty\in
\SL_2(\ad_\infty),\; u\in {\bar \G},\; c\in \SL_2(F)\,.
\end{equation}
Here we use that $\SL_2(\ad_f) = \SL_2(F) \bar\G$, a consequence of
strong approximation for $\SL_2$. So $f_a$ transforms on the right
according to the character $\hat \ch^{-1}$ of $\bar\G$, and is
left-invariant under $\SL_2(F)$. Square integrability of $f$ on
$\G\backslash G$ is equivalent to square integrability of $f_a$ on
$\SL_2(F) \backslash \SL_2(\ad)$. At the finite places we normalize
the Haar measure on $\SL_2(F\lw\prm)$ so that $\bar\G_\prm$ has
volume~$1$.

\section{Hecke algebra}\label{sect-Ha}

One of the advantages of adele groups is the possibility to view the
classical Hecke operators as convolution operators. In this section
we study the structure of a ring of Hecke operators, locally in
\S\ref{sect-lha}, globally in~\S\ref{sssect-gho}.
\medskip

For $\G=\G_0(I)$ and the character~$\ch$, consider 
the convolution algebra $\he$ of compactly supported functions $\ps$
on $\SL_2(\ad_f)$ satisfying
$\ps(u_1gu_2) = \hat\ch(u_1)^{-1}\,\allowbreak \ps(g)\,
\allowbreak \hat\ch(u_2)^{-1}$ for $u_1,u_2,\in\bar\G$ and
$g\in \SL_2(\ad_f)$. Convolution gives not only a multiplication on
$\he$, but also an action of $\he$ on $(\G,\ch)$-automorphic
functions~$f_a$:
\begin{equation}\label{conv1}
f_a \ast \ps (g) \= \int_{\SL_2(\ad_f)} f_a
(gx^{-1})\, \ps(x)\, dx \= \int_{\SL_2(\ad_f)} f_a(g_\infty,x) \,
\ps(x^{-1}g_f )\, dx \,,
\end{equation}
with $g=(g_\infty, g_f)\in \SL_2(\ad) = G\times\SL_2(\ad_f)$. The
support of a non-zero $\ps\in \he$ is a non-empty compact subset of
$\SL_2(\ad_f)$ that satisfies
$\bar\G \,\supp (\ps) \, \bar\G = \supp(\ps)$. By compactness, the
set $\bar\G \backslash \supp(\ps)$ is finite. Strong approximation
allows us to pick representatives of $\bar\G \backslash \supp \ps$ in
$\SL_2(F)$. We normalize the Haar measure on $\SL_2(\ad_f)$ by giving
$\bar\G$ volume~$1$. Then we have for $g_\infty \in G$:
\begin{equation}\label{conv2}
\begin{aligned}
f_a \ast \ps (g_\infty,1) &\= \sum_{\x\in \bar\G \backslash \supp \ps}
f_a(g_\infty,\x^{-1})\,\ps(\x) \\
& \= \sum_{\x\in \G\backslash (\SL_2(F)\cap\supp \ps)} f_a\bigl(
\x^{-1}(\x g_\infty,1)\bigr)\, \ps(\x)\\
& \= \sum_{\x\in \G\backslash (\SL_2(F)\cap\supp \ps)} f(\x g_\infty)
\,\ps(\x)\,.
\end{aligned}
\end{equation}
Thus we obtain the action of $\ps$ in terms of a finite sum of left
translates of the $(\G,\ch)$-automorphic function~$f$.

If $f_a\in L^2\bigl( \SL_2(F) \backslash \SL_2(\ad) \bigr)$, then so
is $F_\x:g\mapsto f_a\bigl(g (1,\x^{-1}) \bigr)\, \ps(\x)$ for all
$\x\in \SL_2(\ad_f)$, and $\|F_\x\|_2 = \|f_a\|_2\; |\ps(\x)|$. Hence
the convolution operator $f_a\mapsto f_a\ast\ps$ defines a bounded
operator on the Hilbert space $L^2\bigl( \SL_2(F) \backslash
\SL_2(\ad) \bigr)$ with norm at most
\begin{equation}\label{normbd}
\|\ps\|_\infty \cdot \# \bigl( \bar\G \backslash \supp \ps\bigr)\,.
\end{equation}
For each $x\in \SL_2(\ad_f)$, right translation $R_x$ given by $(R_x
f)(g) = f(gx)$ is a unitary operator in $L^2\bigl( \SL_2(F)
\backslash \SL_2(\ad) \bigr)$. We put $\ps^\ast(x) =
\overline{\ps(x^{-1})}$. If $\ps$ as above satisfies $\ps^\ast=\ps$,
then convolution by $\ps$ defines a symmetric bounded operator on
$L^2\bigl( \SL_2(F) \backslash \SL_2(\ad) \bigr)$.

\subsection{Local Hecke algebra}\label{sect-lha}
Before defining the subalgebra of $\he$ with which we will work, we
will consider some pertinent local convolution algebras, for primes
outside~$I$. We recall that the character $\hat \ch$ is trivial on
$\SL_2(F\lw\prm)$ for all $\prm\nmid I$.

\begin{defn}For each finite prime~$\prm$ not dividing~$I$ we denote by
$\he_\prm$ the convolution algebra of compactly supported functions
on $\SL_2(F_\prm)$ that are left and right invariant under
translation by elements of $\SL_2(\Ocal_\prm)$.
\end{defn}
We recall the well known description of the structure of~$\he_\prm$.
Let $\hat \pi_\prm$ be a generator of the maximal ideal $\bar\prm$ of
$\Ocal_\prm$.

The unit element of $\he_\prm$ is the characteristic function
$\e_\prm$ of $\bar\G_\prm = \SL_2(\Ocal_\prm)$.

A basis of $\he_\prm$ is given by the
characteristic functions $\heop{\prm^{2k}}$ of the sets
\begin{equation}\label{Dtdecom}
\Dt(\prm^{2k}) \= \left\{ \hat\pi_\prm^{-k}g\;:\; g\in
M_2(\Ocal_\prm),\, \det g \= \hat\pi_\prm^{2k}\right\}\,,
\end{equation}
for $k\in \NN_0$. So $\Dt(\prm^0)=\e_\prm$.
We have
\begin{align}\label{Dtdecomp}
\Dt(\prm^{2k}) &\= \bigsqcup_{l=0}^{2k}
\bigsqcup_{b\in\Ocal_\prm/\bar\prm^l} \bar\G_\prm
\matc{\hat\pi_\prm^{k-l}} {b\hat\pi_\prm^{-k} }0 {\hat\pi_\prm^{l-k}
}\\
\nonumber
&\= \bigsqcup_{l=0}^{2k} \bigsqcup_{\bt\in
\bar\prm^{l-2k}/\bar\prm^{2l-2k}} \bar\G_\prm h\bigl(
\hat\pi_\prm^{k-l}\bigr)
n(\bt) \,,
\end{align}
in the notations of~\S\ref{sect-not}. For $k\geq 1$ we have the
relation:
\begin{equation}\label{Hrelation}
\heop{\prm^{2k}} \ast \heop{\prm^2} \= \heop{\prm^{2k+2}}
+ \nrm\prm\, \heop{\prm^{2k}} + \nrm\prm^2\, \heop{\prm^{2k-2}}\,.
\end{equation}
This implies that $\he_\prm$ is a polynomial ring in the variable
$\heop{\prm^2}$. Furthermore, \eqref{Hrelation} allows us to check
that $\he_\prm$ is isomorphic to the subring of the ring
$\CC[X_\prm,X_\prm^{-1}]$ of even Laurent polynomials that are
invariant under $X_\prm\mapsto X_\prm^{-1}$, by sending
$\heop{\prm^{2k}}$ to
\begin{equation}\label{Tpk}
\nrm\prm^k \sum_{j=0}^{2k} X_\prm^{2k-2j}\,.
\end{equation}
Each character of $\he_\prm$ is uniquely determined by a map
$X_\prm \mapsto \nrm\prm^{\nu_\prm}$, where the parameter
\begin{equation}\label{nuprm}
\nu_\prm \in \CC\bmod \frac{\pi i}{\log \nrm\prm} \ZZ\end{equation}
is determined up to $\nu_\prm\mapsto -\nu_\prm$. This parameter is
related (but not equal to) the usual Satake parameter in
$\CC\bmod \frac{2\pi i}{\log \nrm\prm}\ZZ$ in the case when the Hecke
operator $\heop\prm$ corresponding to
$\sqrt{\nrm\prm}\,(X_\prm+\nobreak X_\prm^{-1})$ can be defined.

Thus we have described the structure of $\he_\prm$ completely for
primes $\prm\dividesnot I$. For primes $\prm$ dividing $I$ the
algebra $\he_\prm$ is the convolution algebra of compactly supported
functions on $\SL_2(F_\prm)$ that transform on the left and on the
right by the character $\hat\ch_\prm^{-1}$ of $\bar\G_\prm$. We do
not go into its structure, and only note that its unit element
$\e_\prm$ is equal to $\hat\ch_\prm^{-1}$ on $\bar\G_\prm$ and is
zero outside $\bar\G_\prm$.

\subsection{Global Hecke algebra}\label{sssect-gho}
We shall work with the following subring of~$\he$:
\begin{defn} We denote by $\he^I$ the convolution subalgebra of $\he$
of compactly supported bi-$\hat\ch^{-1}$-equivariant functions on
$\SL_2(\ad_f)$ spanned by
\[ \mathop{\textstyle\bigotimes}_{\prm\dividesnot I}\ps_\prm \,
\otimes\, \mathop{\textstyle\bigotimes}_{\prm\divides I} \e_\prm\,,\]
where $\ps_\prm\in \he_\prm$ for all $\prm\dividesnot I$, and
$\ps_\prm=\e_\prm$ for all but finitely many~$\prm$.
\end{defn}
We denote by $\otimes_\prm \ps_\prm$ the function
$g\mapsto \prod_\prm \ps_\prm(g_\prm)$ on the group $\SL_2(\ad_f)$.
The condition that $\ps_\prm=\e_\prm$ for almost all~$\prm$ ensures
that $ \otimes_\prm \ps_\prm$ makes sense. The algebra $\he^I$ is
isomorphic to the (restricted) tensor product of the algebras
$\he_\prm$ with $\prm\dividesnot I$. It is a commutative algebra. In
this way we restrict our attention to the places where the local
Hecke algebra is unramified.

Next we build elements of $\he$ from elements of $\he_\prm$ with
$\prm\nmid I$. Since the character $\hat \ch$ may be non-trivial this
requires some care. We also describe the action of elements of $\he$
on automorphic forms for $\G_0(I)$.\medskip

We consider a nonzero ideal $\aid$ in $\Ocal$ prime to~$I$. It has the
 form $\aid = \prod_{\prm\in P} \prm^{k_\prm}$, with $P$ a finite set
of prime ideals not dividing~$I$, and positive $k_\prm$. We define
\begin{equation}\label{Ta2d}
\heop{\aid^2} \= 
\mathop{\textstyle\bigotimes}_{\prm\in P} \heop{\prm^{2k_\prm}}
\otimes \mathop{\textstyle\bigotimes}_{\prm\not\in P}\e_\prm\,.
\end{equation}
Thus defined, convolution by $\heop{\aid^2}$ satisfies
$\heop{\aid^2}^\ast \= \heop{\aid^2}$.

With $\aid$ and $P$ as above, we consider the action of
$\heop{\aid^2}$ on automorphic forms, first under the assumption that
all $\prm \in P $ are principal ideals in $\Ocal$.

So $\prm=(\pi_\prm)$. Then
 $\pi_\prm \in \hat\pi_\prm \Ocal_\prm^\ast$, and
$\pi_\prm \in \Ocal_\qrm^\ast$ at all other finite places~$\qrm$. The
element $\hat\pi_\prm$ embedded in~$\ad_f^\ast$ by taking~$1$ at all
places not equal to~$\prm$, satisfies
$\hat\ch\bigl(\hat\pi_\prm\bigr)=1$. However, the representative
$\pi_\prm\in F^\ast$ need not be in the kernel of~$\ch$, so we may
have $\hat\ch(\pi_\prm)\neq 0$.

The Iwasawa decomposition of $\SL_2(F_\prm)$ shows that we can choose
representatives $\hat x$ of $\bar\G \backslash \supp \heop{\aid^2}$
of the form
$\hat x=\matc q00{1/q} \matc 1{\hat\bt}01 \in \SL_2(\ad_f)$, with
\begin{alignat}3 \label{qbt}
q_\prm &\= \hat\pi_\prm^{k_\prm-l_\prm}\,,&\quad \hat\bt_\prm &\in
\bar\prm^{l_\prm-2k_\prm}
&\qquad&\text{if }\prm\in P\,,\\
\nonumber
q_\qrm &\= 1 \,,& \hat \bt_\qrm&\=0&&\text{if }\qrm\not\in P\,,
\end{alignat}
where each $l_\prm$ runs from $0$ to $2k_\prm$, and where $\hat
\bt_\prm$ runs through a system of representatives of
$\bar\prm^{l_\prm-2k_\prm}/\bar\prm^{2l_\prm-2k_\prm}$. The
representatives $\hat x$ need not be in $\SL_2(F)$. We take
\[ x\=\matc {q_\bid}00{1/q_\bid} \matc 1\bt01 \in \SL_2(F)\]
such that $x\in \bar\G \hat x$ with
$\bid = \prod_{\prm\in P} \prm^{l_\prm}$,
$q_\bid= \prod_{\prm\in P} \pi_\prm^{k_\prm-l_\prm}$, and $\bt$ in a
system of representatives of $\bid\,\aid^{-2} / \bid^2\, \aid^{-2}$.
In this way for $g \in \SL_2(\ad)$
\begin{equation}
\bigl( f_a\ast \heop{\aid^2}\bigr)(g) \= \sum_{\bid\divides \aid^2}
\sum_{\bt \in \bid\aid^{-2} / \bid^2 \aid^{-2}} \hat \ch(q_\bid)
f_a\left( g \left( 1, n(\bt)^{-1} h(q_\bid)^{-1} \right) \right)\,,
\end{equation}
which becomes in terms of $(\G,\ch)$-automorphic functions
on~$\G\backslash G$:
\begin{equation}\label{He-r-p}
\bigl( f|\heop{\aid^2} \bigr)(g_\infty)\= \sum_{\bid\divides \aid^2}
\sum_{\bt \in \bid\aid^{-2} / \bid^2 \aid^{-2}} \hat \ch(q_\bid) f
\left( h(q_\bid) n(\bt)
g_\infty\right)\qquad(g_\infty\in G)\,.
\end{equation}

If we can choose $\pi_\prm$ totally positive for all $\prm\in P$, a
description of $\heop\aid$ similar to the latter formula is possible
for all ideals $\aid$ built with prime ideals in~$P$.\smallskip

In general, not all prime ideals $\prm\in P$ are principal, so we
proceed as follows. Taking $\bt\in \bid\aid^{-2} \subset F$,
$\bt \in \hat \bt+\bid^2\aid^{-2}$, we have in the notations
of~\eqref{qbt}:
\[ h(q) n(\hat \bt) \in \bar\G h(q)n(\bt) \,. \]
By strong approximation there exists $g_\bid\in \SL_2(F)$ such that
\begin{equation}\label{gbt}
h(q)n(\hat \bt) \in \bar\G g_\bid \matc 1\bt01\,.
\end{equation}
Indeed, for $\bid$ such that $\aid\,\bid^{-1}$ is principal, we may
take $g_\bid$ to be a diagonal matrix, in particular,
$g_\Ocal = \matc 1001$. For other~$\bid$ we use that in any field we
have
\[ \matc t00{1/t} \= \matc 1t01 \matc10{-t^{-1}}1 \matc 1t01
\matr0{-1}10\qquad(t\neq0)\,.\]
For any ideal $J$ there are elements $\x,\eta\in F$ and
$u,v\in \bar J$, such that $q=\x+u$, $q^{-1}=\eta+v$ for $q$ as
in~\eqref{qbt}. If we take $J$ as the product of sufficiently high
(depending on $q$) powers of primes in~$P$ and $I$, we have
\[ h(q) \= \matc 1{\x+u}01 \matc10{-\eta-v}1 \matc1{\x+u}01
\matr0{-1}10 \in \bar \G \matc 1\x01\matc10{-\eta}1 \matc
1\x01\matr0{-1}10\,,\]
so that $h(q) \in \bar\G \cdot\SL_2(F)$.
This leads to the following description of convolution by
$\heop{\aid^2}$:
\begin{equation}
\bigl( f_a \ast \heop{\aid^2} \bigr)(g) \= \sum_{\bid\divides \aid^2}
\hat \ch( g_\bid)^{-1} \sum_{\bt \in \bid\aid^{-2}/\bid^2\aid^{-2}}
f_a\left( g,\left(1, n(\bt)^{-1} g_\bid^{-1} \right)
\right)\,,
\end{equation}
for any $g \in \SL_2(\ad)$. Since the $g_\bid n(\bt)$ in this sum are
elements of $\SL_2(F)$, we can go over to the corresponding function
on $G$, and obtain:
\begin{prop}\label{prop-Heop}Let $\aid$ be a non-zero integral ideal
in $\Ocal$ prime to~$I$. For each integral ideal $\bid$ dividing
$\aid$ there exists $g_\bid\in \SL_2(F)$ such that for any
$(\G,\ch)$-automorphic function $f$ on~$G$, and for any $g \in G$:
\begin{equation}\label{Ta2}
\bigl( f|\heop{\aid^2} \bigr)(g) \= \sum_{\bid\divides \aid^2} \hat
\ch( g_\bid)^{-1} \sum_{\bt \in \bid\aid^{-2}/\bid^2\aid^{-2}}
f\left( g_\bid n(\bt) g\right)\,.
\end{equation}

In the sum, $\bt$ runs over elements of $\bid\aid^{-2}\subset F$
representing the classes modulo $\bid^2\aid^2$. The elements $g_\bid$
are in $\SL_2(F)$. For $\bid=\Ocal$ we can take
$g_\bid = \mathrm{Id}$. 
If all prime ideals dividing $\aid$ are principal we can take all
$g_\bid$ to be diagonal matrices satisfying~\eqref{gbt}.
\end{prop}
In particular, the Hecke operators act in the space
$L^2_\x(\G\backslash G,\ch)$ as bounded operators.
(See~\eqref{normbd}.) The relation $T(\aid^2)^\ast = T(\aid^2)$
implies that $T(\aid^2)$ acts as a symmetric bounded operator. Its
norm is bounded by the number of terms in~\eqref{Ta2}.

\section{Distribution of eigenvalues of Hecke
operators}\label{sect-ad}
It will be critical for this paper to work out the relation between
the eigenvalues of Hecke operators and the Fourier coefficients of
automorphic forms. In the first subsection we discuss the action of
Hecke operators on Fourier expansions (Propositions
\ref{prop-Fourterm} and~\ref{prop-He0}) and then we apply these
results to give an expression for the eigenvalues of Hecke operators
on automorphic cusp forms (Theorem~\ref{thm-ew-pi}). In the final
subsection we give a proof of our main result, Theorem~\ref{thm-afH}
in several steps. The main tools in the proof are
Theorems~\ref{thm-af'} and~\ref{thm-ew-pi}.

\subsection{Hecke operators and Fourier expansion}

\begin{lem}\label{lem-fb}Let $\aid$ be a non-zero integral ideal in
$\Ocal$ relatively prime to~$I$, and let $\bid\divides \aid^2$. Then
there is a unique $\k\in \Pcal$ such that the element~$g_\bid$ in
Proposition~\ref{prop-Heop} has the form $g_\bid=\g g_\k p$ with
$\g\in \G$ and $p=n(b)\,h(a) \in \SL_2(F)$.

For each $(\G,\ch)$-automorphic function $f$ on~$G$ the function
\begin{equation}
\label{fb}
f_\bid: g \mapsto \sum_{\bt\in \bid\aid^{-2}/\bid^2\aid^{-2}} f(g_\bid
\, n(\bt)\, g)
\end{equation}
is left-invariant under $\GmN\infty$. The Fourier terms of $f_\bid$
are given by
\begin{equation}\label{Ffb}
F_{\infty,r} f_\bid(g) \= \begin{cases}
\nrm\bid\; \ch(\g) \, F_{\k,a^{-2}r} f(pg)&\text{ for }r\in
\bid^{-1}\aid^2 \Ocal'\,,\\
0&\text{ for other }r\in \Ocal'\,.
\end{cases}
\end{equation}
\end{lem}
\begin{proof}The element $g_\bid\in \SL_2(F)$ is in $\bar \G \, h(q)$
with $q$ as in~\eqref{qbt}. Since $g_\bid\,\infty$ is a cusp, it is
of the form $\g \k$ with $\g\in \G$, for a unique $\k\in \Pcal$. Then
$g_\k^{-1}\g^{-1}g_\bid \in \SL_2(F)$ fixes $\infty$, and is hence of
the form $p=n(b)h(a)$ with $b\in F$, $a\in F^\ast$.

We claim that
\begin{equation}\label{M-expr}
a^2 \bid^2 \aid^{-2} \= M_\k\,.
\end{equation}

Indeed, the construction of the system of representatives in
\S\ref{sssect-gho} implies that for different $\bt$ in a system of
representatives of $\bid\aid^{-2}/\bid^2\aid^{-2}$ the sets
$\G\g_\bid n(\bt)$ are disjoint, and furthermore $g_\bid n(\bt)
g_\bid^{-1}\in \G$ if $\bt \in \bid^2\aid^{-2}$. Hence, for all $\bt
\in \bid^2\aid^{-2}$ we have $g_\k pn(\bt) p^{-1} g_\k^{-1}\in \G$,
which implies that $n(a^2\bt)\in g_\k^{-1} \GmN\k g_\k$ for all
such~$\bt$. This shows that $a^2 \bid^2 \aid^{-2} \subset M_\k$.

Conversely, if $x\in M_\k$, then $g_\k n(x)g_\k^{-1}\in \G$, and
$g_\bid n(a^{-2}x)g_\bid^{-1} \in \g^{-1}\G\g=\G$. We have
$g_\bid=u h(q)$, with $u\in \bar\G$. Hence
$h(q) n(a^{-2}x)h(q)^{-1}\in \bar\G$, and $q^2 a^{-2}x
\in \bar\Ocal$. Since $q\bar\Ocal = \bar \aid \hat \bid^{-1}$, we
have $x\in a^2 \hat\bid^2 \bar \aid^{-2}\cap F = a^2
\bid^2\aid^{-2}$.

Moreover, for $\bt\in \bid^2\aid^{-2}$, we have
\[\ch\left( g_\k n(a^2\bt)
g_\k^{-1}\right)
= \hat \ch(g_\k p) \hat\ch\bigl(n(\bt\bigr) \hat\ch(g_\k
p)^{-1}=1\,.\]
So $\ch$ is trivial on $\GmN\k$ and $\tilde M_\k'$ as defined
in~\eqref{Mtk-def} is equal to $M_\k' = a^{-2}
\bid^{-2} \aid^2\Ocal'$.

The set $ \bigsqcup_{\bt\in \bid\aid^{-2}/\bid^2\aid^{-2}} \G g_\bid
n(\bt)$ is right-invariant under the group $\bigl\{n(\om)\;:\; \om
\in \bid\aid^{-2}\bigr\}$, which contains $\GmN\infty$. Hence
$f_\bid$ in~\eqref{fb} is left-invariant under $\GmN\infty$. It has a
Fourier expansion with terms of order $r\in \Ocal'$, and the terms
with $r$ outside $(\bid\aid^{-2})'=\bid^{-1}\aid^2\Ocal'$ vanish. For
$r\in \bid^{-1}\aid^2\Ocal'$:
\begin{align*}
F_{\infty,r}&f_\bid(g) \= \frac1{\vol(\RR^d/\bid\aid^{-2})}
\int_{\RR^d/\bid\aid^{-2}} e^{-2\pi i S(rx)}\, f_\bid\left( n(x)
g\right)\,dx\displaybreak[0]
\\
\nonumber
&\= \frac1{\vol(\RR^d/\bid\aid^{-2})} \int_{\RR^d/\bid\aid^{-2}}
\sum_{\bt\in \bid\aid^{-2}/\bid^2\aid^{-2}} e^{-2\pi i
S\bigl(r(x+\bt)\bigr)} \, f\left( \g g_\k p n(x+\bt) g \right)\,
dx\,,
\end{align*}
where we have written $g_\bid=\g g_\k p$, and have used that
$S(r\bt) \in \ZZ$. The function
$x\mapsto e^{-2\pi i S(rx)}  \allowbreak
f\left( \g g_\k p n(x) g \right)$ is a function on
$\RR^d/\bid^2\aid^{-2}$.  Hence the
integration over~$x\in\RR^d/\bid\aid^{-2}$ and the summation
over~$\bt \in \bt\aid^{-2}/\bid^2\aid^{-2}$ is replaced by
integration over $x\in \RR^d/\bid^2\aid^{-2}$:
\[F_{\infty,r} f_\bid(g) \= \frac1{\vol(\RR^d/\bid\aid^{-2})}
\int_{\RR^d/\bid^2\aid^{-2}} e^{-2\pi i S(rx)}\, f\left( \g g_\k p
n(x) g \right)\,dx\,. \]
Using also $p=n(b)h(a)$ and the fact that $f$ is
$(\G,\ch)$-automorphic we obtain
\begin{align*}
F_{\infty,r}&f_\bid(g) \= \frac{\ch(\g)}{\vol(\RR^d/\bid\aid^{-2})}
\int_{\RR^d/\bid^2\aid^{-2}} e^{-2\pi i S(rx)} \, f\left( g_\k n(b)
h(a)
n(x) g\right)\, dx \displaybreak[0]
\\
&\=\frac{\ch(\g)}{\vol(\RR^d/\bid\aid^{-2})}
\int_{\RR^d/\bid^2\aid^{-2}} e^{-2\pi i S(rx)} \, f\left( g_\k
n(a^2x) p g\right)\, dx\,.
\end{align*}
Next we carry out the substitution $x\mapsto a^{-2}x$:
\[ F_{\infty,r}f_\bid(g) \= \frac{\ch(\g)}{|\nrm
a|^2\;\vol(\RR^d/\bid\aid^{-2})} \int_{\RR^d/a^2\bid^2\aid^{-2}}
e^{-2\pi iS(ra^{-2}x)} \, f\left( g_\k n(x) p g\right)\,.\]
 We have seen in~\eqref{M-expr}
that $a^2\bid^2\aid^{-2}=M_\k$. So we obtain the formula
in~\eqref{Ffb} when we show that
\[ \frac1{\bigl|\nrm a \bigr|^2\, \vol(\RR^d/\bid\aid^{-2})} \=
\frac{\nrm \bid}{\vol(\RR^2/M_\k)}\,.\]
This follows from the facts that $\vol\bigl(
\RR^d/\bid\cid\bigr) \= \nrm \bid \,\allowbreak\vol\bigl( \RR^d/\cid\bigr)
$ and $\vol\bigl(\RR^d/a\Ld\bigr) = |\nrm
a|\allowbreak\, \vol\bigl( \RR^d/\Ld\bigr)$, for all fractional
ideals $\bid$ and~$\cid$, all lattices $\Ld$, and all $a\in F^\ast$.
\end{proof}
\begin{lem}\label{lem-fbi} Suppose that $\k=\infty$ in the situation
of Lemma~\ref{lem-fb}. Then we can choose $g_\bid=h(a)$, with
$a\in F^\ast$, and $\g=1$. The Fourier terms of $f_\bid$ are given by
\begin{equation}
F_{\infty,r} f_\bid(g) \= \begin{cases}\nrm\bid\; F_{\infty,
a^{-2}r}f\bigl( h(a)g\bigr)&\text{ if } r\in
\bid^{-1}\aid^2\Ocal'\,,\\
0&\text{ for other }r\in \Ocal'\,.
\end{cases}
\end{equation}
If $\bid=\aid$, then we can further take $a=1$, while if
$\bid\neq\aid$ then $a\in F^\ast\setminus \Ocal^\ast$.
\end{lem}
\begin{proof}We have $g_\bid = \g n(b)h(a) \in \bar \G h(q)$ with $q$
as in~\eqref{qbt}. This implies that $a = u q$ with
$u\in \bar\Ocal^\ast$, and hence the ideal $q\,\bar\Ocal =
\aid\bid^{-1}\bar\Ocal$ is generated by $a\in F^\ast$. So
$\aid\bid^{-1}$ is a principal ideal, $a$ is a generator of
$\aid\bid^{-1}$, and we can choose $g_\bid = h(a)$, $\g=1$. So
$p=h(a)$ and the formulas for the Fourier terms follow from
\eqref{Ffb} in the previous lemma.
\end{proof}

\begin{prop}\label{prop-Fourterm} Let $\aid$ be a non-zero integral
ideal in $\Ocal$ relatively prime to~$I$, and let $f$ be a
$(\G,\ch)$-automorphic function on~$G$. For each non-zero integral
ideal $\bid$ dividing $\aid$ there exist unique $\k(\bid)\in \Pcal$,
$p(\bid) = n\bigl( b(\bid)\bigr)\,
h\bigl( a(\bid) \bigr) \in \SL_2(F)$, and $\g(\bid)\in \G$ such that
for all $r\in \Ocal'$
\begin{align}
F_{\infty,r} \bigl( &f|\heop{\aid^2}\bigr)(g) \\
\nonumber
& \= \sum_{ \bid\divides \aid^2\,,\; r\in \bid^{-1}\aid^2\Ocal' }
\nrm\bid\; \hat \ch\bigl( g_{\k(\bid)}p(\bid)
\bigr)^{-1} \; F_{\k(\bid),a(\bid)^{-2}r} f \bigl( p(\bid) g\bigr)\,.
\end{align}

If $\k(\bid)=\infty$, then either $\bid=\aid$, $a(\bid)=1$,
$b(\bid)=0$, and $\g(\bid)=1$, or $\bid\neq\aid$, and
$a(\bid)\in F^\ast\setminus \Ocal^\ast$.
\end{prop}
\begin{proof}Replacing the inner sum on the right hand side
of~\eqref{Ta2} in Proposition~\ref{prop-Heop} by $f_\bid$
in~\eqref{fb} in Lemma~\ref{lem-fb}, we obtain
\[ \bigl( f|\heop{\aid^2}\bigr) \= \sum_{\bid\divides \aid^2} \hat\ch(
g_\bid)^{-1}\; f_\bid(g)\,.\]
Applying Lemma~\ref{lem-fb} for each $\bid\divides \aid^2$ we obtain
$\k(\bid)$, $p(\bid)$ and $\g(\bid)$, now explicitly depending
on~$\bid$. For a given $\bid$ we obtain a non-zero contribution to
the Fourier term of order $r\in \Ocal'$ only if
$r\in \bid^{-1}\aid^2 \Ocal'$ (see \eqref{Ffb} in
Lemma~\ref{lem-fb}), and that contribution is
\begin{align*} \hat\ch( g_\bid)^{-1} \, \nrm\bid\,& \ch\bigl( \g(\bid)
\bigr)
\, F_{\k(\bid),a(\bid)^{-2}r}f\bigl( p(\bid)g\bigr) \\
&\= \nrm\bid\, \hat\ch\bigl(g_{\k(\bid)} p(\bid) \bigr)^{-1} \,
F_{\k(\bid),a(\bid)^{-2}r}f\bigl( p(\bid)g\bigr)\,.\end{align*}
The other statements follow from Lemma~\ref{lem-fbi}.
\end{proof}

We recall that we
defined cuspidality,  after Proposition~\ref{prop-Heop}, as the
vanishing of all Fourier terms $F_{\k,0}$ with $\k\in \Pcal_\ch$. So
the following proposition tells us that the Hecke operators
$T(\aid^2)$ preserve cuspidality.

\begin{prop}\label{prop-He0}Let $f$ be a $(\G,\ch)$-automorphic
continuous function on~$G$ for which $F_{\k,0}f=0$ for all
$\k\in \Pcal_\ch$. Then
\[ F_{\k,0}\bigl( f|\heop{\aid^2}\bigr)\=0\qquad\text{ for all }\k\in
\Pcal_\ch\,,\]
for each non-zero ideal $\aid $ in $\Ocal$ that is relatively prime
to~$I$.
\end{prop}
\begin{proof}Let $f$ and $\aid$ be as in the proposition. Consider
$\k\in \Pcal_\ch$. Then
\begin{align*}
F_{\k,0} \bigl(f|\heop{\aid^2}\bigr)(g) &\= \sum_{\bid\divides \aid}
\frac{\hat\ch(g_\bid)^{-1}}{\vol(\RR^d/M_{\k})}
\\
&\qquad\hbox{} \cdot
\int_{\RR^d/M_{\k}} \sum_{\bt \in \bid\aid^{-2}/\bid^2\aid^{-2}} f(
g_\bid n(\bt) g_{\k} n(x)g)\, dx\displaybreak[0]\\
&\= \int_{\RR^d/M_{\k} } \sum_j c_j \,f(h_j n(x)\, g)\,dx \,,
\end{align*}
where $h_j$ runs over a finite set of elements of $\SL_2(F)$, and
$c_j\in \CC$. We note that the function
$x\mapsto \sum_j c_j f(h_j n(x)\, g)$ on $\RR^d$ is $M_\k$-periodic,
but the individual terms may not be $M_\k$-periodic. Since the
finitely many $h_j$ are all in $\SL_2(F)$, there is a fractional
ideal $\Ld_0$ with finite index in $M_\k$ such that the individual
terms are $\Ld_0$-periodic.

Write $h_j=\g_j g_{\k_j} n(b_j) h(a_j)$ with $g_j \in \G$,
$\k_j\in\Pcal$ and $p_j =\matc {a_j} {b_j}0{1/a_j} \in \SL_2(F)$.
Then for all fractional ideals $\Ld$ in~$\Ld_0$ we have
\begin{equation}
\label{ih}
\begin{aligned}
\int_{\RR^d/\Ld} &f\bigl(h_j n(x) g\bigr)\, dx \= \ch(\g_j)
\int_{\RR^d/\Ld} f\bigl(g_{\k_j} n(b_j)h(a_j)n(x)g\bigr)\, dx\\
&\= \ch(\g_j)\, \bigl|\nrm{a_j}\bigr| \, \int_{\RR^d/a_j\Ld} f \bigl(
g_{\k_j} n(x)h(a_j)g\bigr)\, dx\,.
\end{aligned}
\end{equation}

For each $\k'\in \Pcal$ the fractional ideal $M_{\k'}$ contains a
fractional ideal $\hat M_{\k'}$ on which $\ch$ is trivial. (For
$\k'\in \Pcal_\ch$ we can take $\hat M_{\k'}=M_{\k'}$.) The
assumptions of the proposition imply that for all $g\in G$:
\[ \int_{\RR^d/\hat M_{\k'}} f\bigl( g_{\k'}n(x)g\bigr)\, dx\=0\,.\]
Taking for $\Ld\subset\Ld_0$ a non-zero ideal in ~$\Ocal$ divisible by
all primes that contribute denomintors of the~$a_j$, we can arrange
that $a_j \Ld \subset \hat M_{\k_j}$ for all $j$. Thus, we conclude
that the integral in~\eqref{ih} vanishes. So $F_{\k,0} \bigl(
f|\heop{\aid^2} \bigr)=0$.
\end{proof}

\subsection{Action on cusp forms} We turn to the action of the Hecke
operators on the cuspidal space
$L_\x^{2\mathrm{cusp}}(\G\backslash G,\ch)$, and define the
eigenvalues of Hecke operators that occur in Theorem~\ref{thm-afH}.
\medskip

Since all $\heop{\aid^2}$ act as self-adjoint
bounded operators and they all
commute with the Casimir operators $C_j$, $j=1,\ldots,d$, we can
arrange the orthogonal system $\{ V_\varpi\}$ of irreducible cuspidal
subspaces so that $V_\varpi|\he \subset V_\varpi$ for each~$\varpi$.
Hence we have for each $\varpi$ a character $\ch_\varpi$ of~$\he$
that gives the eigenvalue of the Hecke operators on the irreducible
space~$V_\varpi$. Since the convolution operator determined by
$\heop{\aid^2}$ is symmetric, the value
$\ch_\varpi\bigl(\heop{\aid^2}\bigr)$ is real for all $\aid$ prime
to~$I$. By \eqref{normbd}, \eqref{Dtdecom} and \eqref{Ta2d} we have
\begin{equation}
\left| \ch_\varpi\bigl( \heop{\aid^2}\bigr)
\right|\leq \prod_{\prm\in P} \#\bigl( \bar\G_\prm \backslash
\Dt(\prm^{2k_\prm})
\bigr) \= \prod_{\prm\in P} \sum_{j=0}^{2k} \nrm\prm^j\, .
\end{equation}
In \eqref{nuprm} we have assigned a parameter $\nu_\prm$ to each
character of a local Hecke algebra $\he_\prm$. Thus, to $\ch_\varpi$
corresponds, at the place $\prm$ outside~$I$, a parameter
\begin{equation}
\nu_{\varpi,\prm} \in i\biggl[ 0,\frac{\pi}{2\log \nrm\prm} \biggr]
\cup \biggl( 0,\frac12\biggr]\,.
\end{equation}
By \eqref{Tpk} we have
\begin{equation}
\ch_\varpi\bigl( \heop{\prm^{2k}}\bigr) \= \nrm\prm^k \sum_{j=0}^{2k}
\nrm\prm^{2(k-j)\nu_{\varpi,\prm}}\,,
\end{equation}
which corresponds to
\begin{equation}
\label{lbd}
\ld_{\varpi,\prm}^2 + \nrm \prm \;\in\; \bigl[ 9,1+\nrm \prm +
\nrm\prm^2\bigr]\;.
\end{equation}
With this choice, the parameter $\ld_{\varpi,\prm}$ in~\eqref{ldprm}
in the introduction, defined so that $\ld_{\varpi,\prm}^2 - \nrm\prm$
is the eigenvalue of $\heop{\prm^2}$ on $V_\varpi$, is given by
\begin{equation}\label{ldpbd}
\ld_{\varpi,\prm} \= \sqrt{\nrm\prm}\, \bigl(
\nrm\prm^{\nu_{\varpi,\prm}} + \nrm\prm^{-\nu_{\varpi,\prm}} \bigr)
\in [0,1+\nrm\prm]\,.
\end{equation}
We use $\ld_{\varpi,\prm}$ as the main quantity. Note that if the
operator $\heop\prm$ can be defined, the system $\bigl\{ V_\varpi
\bigr\}$ can be rearranged so that $\heop\prm$ acts by $\pm
\ld_{\varpi,\prm} \cdot\mathrm{Id}$
on~$V_\varpi$.

If $\aid=\prod_{\prm\in P} \prm^{k_\prm}$ is prime to~$I$ then
\begin{equation}\label{chpiA}
\ch_\varpi\bigl( \heop{\aid^2} \bigr) \= \prod_{\prm \in P}
S\lw{\prm,2k_\prm}( \ld_{\varpi,\prm})\,,
\end{equation}
where $S\lw{\prm,2k}$ is the only even polynomial of degree~$2k$ such
that
\begin{equation}
S\lw{\prm,2k}\bigl(\sqrt{\nrm\prm}\, (X+X^{-1} \bigr) \= \nrm\prm^k
\sum_{j=0}^{2k} X^{2(k-j)}\,.
\end{equation}

Now we are in a position to give the relation between the eigenvalues
$\ch_\varpi\bigl(\heop{\aid^2}\bigr)$ and the Fourier coefficients of
the cuspidal automorphic representation~$\varpi$:
\begin{thm}\label{thm-ew-pi}Let $\aid$ be a non-zero integral ideal
in~$\Ocal$ relatively prime to~$I$, and let
$r\in \Ocal'\setminus\{0\}$. With the notations of
Proposition~\ref{prop-Fourterm} we have for each irreducible cuspidal
space $V_\varpi $ invariant under the Casimir operators and the Hecke
operators $\heop{\prm^2}$ with $\prm\dividesnot I$ the following
relation for each non-zero ideal $\aid$ in $\Ocal$ prime to~$I$:
\begin{align}\label{Ta2c}
 \ch_\varpi\bigl(&\heop{\aid^2}\bigr) c^{\infty,r}(\varpi) \=
 \sum_{\bid\divides \aid^2\,,\; r\in \bid^{-1}\aid^2\Ocal'} \nrm\bid\;
 \hat\ch\bigl(g_{\k(\bid)} p(\bid)\bigr)^{-1}\, |\nrm{a(\bid)}|\\
\nonumber
&\qquad\hbox{} \cdot
\left( \prod_j\sign\bigl(a_j(\bid)
\bigr)^{\x_j} \right) \, e^{2\pi i S(r b(\bid)/a(\bid)^3)}
\,c^{\k(\bid),r/a(\bid)^2}(\varpi)\,.
\end{align}
\end{thm}
\begin{proof}We pick a weight $q$ occurring in~$V_\varpi$ and
use~\eqref{Fourc-def}:
\begin{align*}
\ch_\varpi\bigl(&\heop{\aid^2}\bigr) \,c^{\infty,r}(\varpi)\,
d^r(q,\nu_\varpi)\, W_q(r,\nu_\varpi;g)\displaybreak[0]\\
&\= \sum_{\bid\divides \aid^2\,,\; r\in \bid^{-1}\aid^2\Ocal'}
\nrm\bid\; \hat\ch\bigl(g_{\k(\bid)} p(\bid)\bigr)^{-1} \,
c^{\k(\bid),r/a(\bid)^2}(\varpi)\; d^{r/a(\bid)^2}(q,\nu_\varpi)\\
&\qquad\hbox{} \cdot
 W_q(r\, a(\bid)^{-2},\nu_\varpi; p(\bid)
 g)\,.\end{align*}
Formulas (1.12) and (2.28) in~\cite{BM9} imply that this is equal to
\begin{align*}
 & \sum_{\bid\divides \aid^2\,,\; r\in \bid^{-1}\aid^2\Ocal'}
 \nrm\bid\; \hat\ch\bigl(g_{\k(\bid)} p(\bid)\bigr)^{-1} \,
 c^{\k(\bid),r/a(\bid)^2}(\varpi)\, \bigl|\nrm{a(\bid)})\bigr|
\\
&\qquad\hbox{} \cdot d^r(q,\nu_\varpi) \, e^{2\pi i
S(rb(\bid)/a(\bid)^3)}\, W_q(r, \nu_\varpi; g)\,
\prod_j\sign\bigl(a_j(\bid)\bigr)^{\x_j}\,.
\end{align*}
This yields statement \eqref{Ta2c} in the theorem.
\end{proof}
\begin{rmk}This theorem generalizes the classical relation between
eigenvalues of Hecke operators $T_p$ with $p$ prime on a cuspidal
eigenform and the Fourier coefficient of order~$p$ of that form. In
the classical context one uses the normalization of the eigenform by
taking its Fourier coefficient of order~$1$ equal to~$1$. This
normalization does not extend to the present situation in a
straightforward way, since there is in general no obvious Fourier
term order $r$ in $\Ocal'\setminus\{0\}$ to play the role of~$1$.
Hence we give a formulation in which $r\in \Ocal'\setminus\{0\}$ can
be chosen freely.
\end{rmk}

\subsection{Proof of Theorem~\ref{thm-afH}} In this subsection, we
give a proof of Theorem~\ref{thm-afH} in two steps. First we prove
Proposition~\ref{prop-asfp}, which is a version of
Theorem~\ref{thm-afH} with the characteristic functions of the
intervals $J_\prm$ replaced by polynomials. Next, \S\ref{sect-extas}
gives the extension of this result to characteristic functions of
intervals.

\subsubsection{Asymptotic formula for polynomials}
\begin{lem}Let $P$ be a finite set of primes of~$F$. Then there are
elements $r\in \Ocal'$ such that $r\not\in \prm \Ocal'$ for all
$\prm\in P$.
\end{lem}
\begin{proof}
For any fractional ideal $\aid$ with prime decomposition
$\aid=\prod_\prm \prm^{a_\prm}$ the quotient $\aid/\bigl( \aid
\prod_{\prm\in P} \prm\bigr) \cong \prod_{\prm\in P}
\bigl(\prm^{a_\prm}/\prm^{a_\prm+1}\bigr)$ contains elements $x$ such
that $x_\prm + \prm^{a_\prm+1}\neq \prm^{a_\prm+1}$ for all
$\prm\in P$.
\end{proof}

\begin{prop}\label{prop-asfp}Let $t\mapsto \Om_t$ be a family of sets
in~$\RR^d$ as in~\eqref{Om-gen}. Let $P$ be a finite set of primes
not dividing~$I$, and let $\ld_{\varpi,\prm}$ and $\Phi$ be as in
\eqref{ldpbd} and \eqref{Phi-def} respectively.

Then, for any choice of even polynomials $q_\prm$, $\prm\in P$, and
any $r\in \Ocal'$ such that $r\not\in \prm\Ocal'$ for each
$\prm\in P$, we have as $t\rightarrow\infty$:
\begin{align}\label{asfp}
\sum_{\varpi,\, \ld_\varpi\in \Om_t} \bigl|
&c^{\infty,r}(\varpi)\bigr|^2 \,\prod_{\prm\in P}
q_\prm(\ld_{\varpi,\prm})\\
\nonumber
& \= \frac{2 \sqrt{|D_F|} \, \vol(\G\backslash G)}{(2\pi)^d}\,
\pl(\Om_t)\, \prod_{\prm\in P} \Phi_\prm(q_\prm) \,+\, o\bigl(
V\lw1(\Om_t)
\bigr)\,.
\end{align}
\end{prop}
\begin{proof}Let $\aid=\prod_{\prm\in P}\prm^{k_\prm}$ with
$k_\prm\geq 0$. The sum in~\eqref{Ta2c} is finite, so
Theorems~\ref{thm-ew-pi} and~\ref{thm-af'} give for any
$r\in \Ocal'\setminus\{0\}$:
\begin{align*}
\sum_{\varpi\,,\; \ld_\varpi\in \Om_t} & \ch_\varpi\bigl(\heop{\aid^2}
\bigr) \,\bigl| c^{\infty,r}(\varpi )\bigr|^2 \= \sum_{\bid\divides
\aid^2,\, r\in \bid^{-1}\aid^2\Ocal'} \nrm\bid\; \hat
\ch\bigl(g_{\k(\bid)}p(\bid)\bigr)^{-1} \;|\nrm{a(\bid)}|\\
\nonumber
&\qquad\hbox{} \cdot
 \left( \prod_j \sign\bigl(a_j(\bid)
 \bigr)^{\x_j} \right)\, e^{2\pi i S(r\,b(\bid)/a(\bid)^3)}\,
\dt_{\infty,\k(\bid) }\, \dt_\k(r,r/a(\bid)^2)\\
\nonumber
&\qquad\hbox{} \cdot
\frac{2\sqrt{|D_F|}\, \vol(\G\backslash G)}{(2\pi)^d}\, \pl(\Om_t)
+ o\bigl( V\lw1(\Om_t) \bigr)\,.
\end{align*}
In the terms in the sum the factor $\dt_{\infty,\k(\bid)}$ is nonzero
(and equal to $1$) only if $\k(\bid)=\infty$. Then either $\bid=\aid$
and $a(\bid)=1$, or $a(\bid)\in F^\ast \setminus \Ocal^\ast$, by
Lemma~\ref{lem-fbi}. Then we see in \S\ref{sss-dt} that
$\dt_\k\bigl(r,r/a(\bid)^2\bigr)=0$ if $a(\bid)\not\in \Ocal^\ast$,
and $\dt_\k\bigl(r,r/a(\bid)^2\bigr)=1$ in the case $\bid=\aid$.
 Thus, we arrive at
\begin{align}\label{a-dis-r}
\sum_{\varpi\,,\; \ld_\varpi\in \Om_t} & \ch_\varpi\bigl(\heop{\aid^2}
\bigr) \,\bigl| c^{\infty,r}(\varpi )\bigr|^2 \\
\nonumber
&\qquad\hbox{} \=
\begin{cases}\nrm\aid\; \frac{2\sqrt{|D_F|}\, \vol(\G\backslash
G)}{(2\pi)^d}\, \pl(\Om_t)
&\text{ if } r\in \aid\,\Ocal'\\
0&\text{ otherwise}\end{cases}\\
\nonumber
&\qquad\qquad\hbox{}+ o\bigl(V\lw1(\Om_t)
\bigr)\,.
\end{align}

We use \eqref{chpiA}, and note that $S_{\prm,0}=1$, to find
\begin{align}\label{Samsp}
&\sum_{\varpi,\, \ld_\varpi\in \Om_t} \bigl|
c^{\infty,r}(\varpi)\bigr|^2 \prod_{\prm\in P} S\lw{\prm,2k_\prm}
(\ld_{\varpi,\prm})
\displaybreak[0]\\
\nonumber
&\qquad\hbox{} \=
\begin{cases}
 \frac{2\, \sqrt{|D_F|}\, \vol(\G\backslash G)}{(2\pi)^d} \pl(\Om_t)
 \prod_{\prm\in P} \nrm\prm^{k_\prm}
&\text{ if } r\in \Ocal'\cdot \prod_{\prm\in P}\prm^{k_\prm}\\
0&\text{ otherwise}
\end{cases}\\ \nonumber \\
\nonumber
&\qquad\qquad\hbox{}+ o\bigl(V\lw1(\Om_t)
\bigr)\,.
\end{align}
We have to consider this for $r\in \Ocal'$ such that $r\not\in \prm
\Ocal'$ for any $\prm\in P$. That means that we obtain only the term
$o\bigl( V\lw1(\Om_t)
\bigr)$ except in the case that all $k_\prm$ vanish.

A computation shows that the measure $\Phi_\prm$ in \eqref{Phi-def}
satisfies
\begin{equation} \Phi_\prm( S\lw{\prm,2k} ) \=
\begin{cases}
1&\text{ if } k=0\,,\\
0&\text{ otherwise}\,.
\end{cases}
\end{equation}
So we can replace the right hand side in~\eqref{Samsp} by that
in~\eqref{asfp} with $q_\prm = S\lw{\prm,2k}$. Since the
$S\lw{\prm,2k}$, $k\geq 0$, form a basis of the even polynomials in
$X_\prm$ , we have completed the proof.
\end{proof}

\subsubsection{Asymptotic formula for characteristic
functions}\label{sect-extas} We complete the proof of
Theorem~\ref{thm-afH} by extension steps also used in \cite{BMP3a}
and~\cite{densII}.

For the families under consideration we can have $\pl(\Om_t)=0$ for
all $t$. That occurs if $E\neq \emptyset$ in~\eqref{Om-gen}, and
$\pl\bigl([A_j,B_j]\bigr)=0$ for at least one $j\in E$. In this case,
equation \eqref{afH} follows directly from~\eqref{asf}.

For all other families under consideration we have
$\pl(\Om_t)\rightarrow\infty$. Then we may view \eqref{asfp} as a
limit formula
\begin{equation}\label{lasf}
\lim_t \mu_t( p ) \= \mu(p)\,,
\end{equation}
for positive measures on the compact space $X_P = \prod_{\prm\in P}
\bigl[ 0,1+\nobreak \nrm\prm\bigr]$:
\begin{align}
\mu_t(f) &\= \frac1{\pl(\Om_t)}\, \sum_{\varpi,\, \ld_\varpi \in
\Om_t} \bigl| c^{\infty,r}(\varpi)
\bigr|^2 f\bigl( (\ld_{\varpi,\prm})_{\prm\in
P}\bigr)\,,\displaybreak[0]\\
\nonumber
\mu(f)&\= \frac{2\sqrt{|D_F|}\, \vol(\G\backslash G)}{(2\pi)^d} \Bigl(
\mathop{\textstyle\bigotimes}_{\prm\in P} \Phi_\prm\Bigr)(f)\,.
\end{align}
Equation \eqref{asfp} gives \eqref{lasf} on tensor products of even
polynomials. By the Stone-Weierstrass theorem we get \eqref{lasf} for
all continuous functions.

Let $J_\prm$ be an interval contained in $\bigl[0,1+\nobreak
\nrm\prm\bigr]$. (According to \eqref{ldpbd} it suffices to consider
intervals of this type.)
For a given $\e>0$, there exist continuous functions $c$ and $C$ on
$X_P$ such that $\mu(C-\nobreak c) \leq \e$, and $0\leq c \leq
\ch\leq C$. From
\[ \begin{array}{ccccc}
\mu_t(c) &\leq & \mu_t (\ch) &\leq & \mu_t(C)\\
\downarrow&&&&\downarrow\\
\mu(c)&\leq&\mu(\ch)&\leq &\mu(C)
\end{array}
\]
and $\mu(C)-\mu(c)\leq \e$ we conclude that
\[ \mu(\ch)-2\e \;\leq \; \liminf_t \mu_t(\ch) \;\leq\; \limsup_t
\mu_t(\ch)
\;\leq \; \mu(\ch)+2\e\,.\]
Since $\e$ is arbitrary, then equation \eqref{lasf} holds for $p=\ch$
and the theorem now follows.

 \appendix
\section{Sum formula}\label{sect-sf}

The proof of \eqref{as} in \cite{densII} is based on the sum formula
in \cite{BM9}. Similarly, the proof of Theorem~\ref{thm-af'} in
\S\ref{sect-af} will be based on a generalization,
Theorem~\ref{thm-sf}, of such a sum formula. In this section we
discuss how to adapt and extend the sum formula in \cite{BM9} to our
present requirements, showing how the proof in {\sl loc.\ cit.}\ can
be modified to give Theorem~\ref{thm-sf}. We shall need an estimate
of generalized Kloosterman sums that is discussed
in~\S\ref{sect-Wte}.

\subsection{Kloosterman sums} The sum formula relates Fourier
coefficients of automorphic representations to Kloosterman sums,
which we discuss now.
\subsubsection{Bruhat decomposition} It is well known that
\begin{align}
\SL_2(F) &\= P_F \sqcup \bc_F\qquad(\text{Bruhat
decomposition})\,,\displaybreak[0]\\
\nonumber
P_F&\=\left\{ \matc \ast\ast0\ast \in \SL_2(F) \right\}\,,\qquad
\bc_F\=\left\{ \matc \ast\ast{\neq0}\ast \in \SL_2(F)\right\}\,.
\end{align}

For $\k,\k'\in \Pcal$ we put
\begin{equation}
\hg{\k'}\Ccal^\k \= \left\{ c\in F^\ast\;:\; \exists \g\in \G\;:\;
g_{\k'}^{-1}\g g_\k \= \matc \ast\ast c\ast\right\}\,,
\end{equation}
and we let $\hg{\k'}\Scal^\k(c)$ denote a system of representatives of
\[ \GmN{\k'} \backslash\left(\G \cap g_{\k'} \bc(c) g_\k^{-1}
\right)/\GmN\k\,,\qquad \bc(c) \= \left\{ \matc \ast\ast c\ast
 \right\}\,. \]
Note that $-\hg{\k'}\Ccal^\k = \hg{\k'}\Ccal^\k$, and that we may use
$\matr{-1}00{-1} \cdot
\hg{\k'}\Scal^\k(c) $ as a system of representatives of
$\GmN{\k'}\backslash\left(\G \cap g_{\k'} \bc(-c) g_\k^{-1}
\right)/\GmN\k$.

\subsubsection{Kloosterman sums}For the present situation the
Kloosterman sums are defined, for $\k,\k'\in \Pcal$, $c\in
\hg{\k'}\Ccal^\k$, $r\in \tilde M_\k'$, $r'\in\tilde M_{\k'}'$, by:
\begin{equation}\label{Kldef}
\kls \ch(\k,r;\k',r';c)\= \sum_{\g\in \hg{\k'}\Scal^\k(c)}
\ch(\g)^{-1} e^{2\pi i S(\frac{r'a+rd}c)}\,,
\end{equation}
where $g_{\k'}\g g_{\k}^{-1} = \matc a \ast cd$. For the cusp $\infty$
this simplifies to a more familiar Kloosterman sum:
\begin{align*} \kls\ch(r,r';c) &\isdef \kls\ch(\infty,r;\infty,r';c)
\\
&\= \sum_{a,d\bmod(c)\,, \; ad\equiv 1\bmod c} \ch\matc abcd^{-1}
e^{2\pi i S
(\frac{r'a+rd}c)}\,,
\end{align*}
with $c\in I$, $c\neq0$, and $r,r'\in \Ocal'$.

Since $\bigl( \hg{\k'} \Scal^\k(c) \bigr)^{-1}$ is a system of
representatives for the double quotient space $\GmN\k \backslash
\bigl(\G \cap g_\k \bc(-c)
g_{\k'}^{-1}\bigr)/\GmN{\k'} $, we have
\begin{equation}\label{Sconj}
\overline{ \kls\ch(\k,r;\k',r';c)} \= \kls\ch(\k',r';\k,r;-c) \=
\ch(-1)
\, \kls\ch(\k',r';\k,r;c)\,,
\end{equation}
and we can use $\hg{\k'}\Scal^\k (c)$ as a system of representatives
$\hg{\k\,}\Scal^{\k'}(c)$.

\subsubsection{Weil type estimate}\label{sect-Wte} The proof of the
sum formula in the form that we will need requires a Weil type
estimate
for the Kloosterman sums occurring in the formula. We will also need
this estimate in the use of the sum formula, when we prove the
asymptotic formula in Theorem~\ref{thm-af'}.

\begin{prop}\label{prop-Kl-est}For $\k,\k'\in \Pcal$, $r\in \tilde
M_\k' \setminus\{0\}$, $r'\in \tilde M_{\k'}' \setminus\{0\}$, there
is a finite set $S$ of prime ideals in $\Ocal_F$ such that, for each
$\e>0$ and for all $c \in \hg{\k'}\Ccal^\k$:
\begin{equation}\label{klest}
 \kls \ch(\k,r;\k',r';c)
\;\ll_{F,I,\k,r,\k',r',\e} \prod_{\prm \in S} \nrm\prm^{v_\prm(c)} \;
\left( \prod_{\prm\not\in S}\nrm \prm^{v_\prm(c)}\right)^{\frac12+\e}
\,,\end{equation}
where $v_\prm$ denotes the valuation at the prime~$\prm$.
\end{prop}
This estimate is weaker than what we may expect to be true. See for
instance the estimate stated in (13)
in~\cite{Venk}.

The proof of~\eqref{klest} is relatively easy. First one establishes a
product formula, reducing the task to estimating local Kloosterman
sums. We put all places where something special happens (places
dividing $I$, places where $\Ocal_\prm
\otimes_\Ocal M_\k\neq \Ocal_\prm$, places where $g_\k\not\in
\SL_2(\Ocal_\prm)$, and similarly for $\k'$) into~$S$, and estimate
the corresponding local Kloosterman sum trivially. At the remaining
places the local Kloosterman sum is a standard one and can be
estimated in the usual way.

\subsection{Sum formula} The sum formula in Theorem~\ref{thm-sf}
relates four terms, each depending on a test function. We discuss
first the test functions and the four terms.

\subsubsection{Test functions} The class of test functions is the same
as in \cite{densII}, \S2.1.1: functions of product type $\ph(\nu) =
\prod_j \ph_j(\nu_j)$, where the factor $\ph_j$ is defined on a strip
$|\re\nu_j|\leq \tau$ with $\frac14<\tau<\frac12$, and on the
discrete set $\frac{1+\x_j}2+\NN_0$. The factor $\ph_j$ satisfies on
its domain $\ph_j(\nu_j) \ll (1+\nobreak |\nu_j|)^{-a}$ for some
$a>2$, and is even and holomorphic on the strip $|\re\nu_j|\leq
\tau$.

The $\nu_j$ occurring in this way are related to spectral data. The
eigenvalues $\ld_{\varpi,j}$ of the Casimir operators in $V_\varpi$
can be written as $\ld_{\varpi, j} = \frac14-\nu_{\varpi,j}^2$, which
we can choose in $(0,\infty)\cup i[0,\infty)$. So the test functions
$\ph$ can be viewed as functions defined on a neighborhood of the set
of possible values of the vectors $\nu_\varpi=(\nu_{\varpi,j})_j$.

\subsubsection{Spectral terms} The sum formula is based on the choice
of two pairs $(\k,r)$ and $(\k',r')$ of cuspidal representatives $\k$
and $\k'$ in $\Ccal$, and non-zero Fourier term orders $r$ and $r'$
at those cusps. In this paper, we need only the case that $r/r'$ is
totally positive.

Closely related to the counting function in \eqref{N'} is the cuspidal
term
\begin{equation}
\nctr (\ph) \= \sum_{\varpi} \overline{c^{\k,r}(\varpi)}\,
c^{\k',r'}(\varpi) \, \ph(\nu_\varpi)\,.
\end{equation}

The Fourier coefficients of Eisenstein series enter into the
Eisenstein term
\begin{align}
\ectr(\ph) &\= 2 \sum_{\ld\in \Pcal} c_\ld \sum_{\mu\in \Ld_{\ld,\ch}
} \int_0^\infty \overline{D^{\k,r}_\x(\ld,\ch;iy,i\mu)}
\\
\nonumber
&\qquad \hbox{} \cdot
D^{\k,r}_\x(\ld,\ch;iy+i\mu)\, \ph(iy+i\mu)\, dy\,,
\end{align}
with $c_\ld > 0$, and $\Ld_{\ld,\ch} $ as in \S2.1.2 of~\cite{densII}.
(The $c_\ld$ differ from those in~\cite{densII} due to the difference
in normalization discussed in~\S\ref{sect-Four-coeff}. Their actual
value is not important for the present purpose.)

\subsubsection{Delta term}\label{sss-dt}The delta term
$\Dt^{\k,r;\k',r'}(\ph)$ can be non-zero only if $\k= \k'$ and $r/r'$
satisfies a strong condition.

We put
\begin{align}\label{dt-def}
 \dt_\k(r,r')&\= \dt_\k( \ch ,\x; r,r' )
\displaybreak[0]\\
\nonumber
& \= \frac12 \sum_{\g \in \GmP\k/\GmN\k \,,\; r/r'=\e^2} \ch(\g) \,
e^{-2\pi i S(r\bt\e)} \prod_j (\sign \e_j)^{\x_j}\,,
\end{align}
where $\g=g_\k \matc {\e} {\bt} 0{\e^{-1}} g_\k^{-1}$. If $\k=\infty$
the $\e$ occurring in $\matc \e \bt 0 {\e^{-1}} \in \G$ are the units
of $\Ocal$. For other $\x\in \Ccal$, the $\e$ occurring in
$g_\k \matc {\e} {\bt} 0{\e^{-1}} g_\k^{-1} \in \G$ form also a
subgroup of $F^\ast$ isomorphic to $(\ZZ/2) \times \ZZ^{d-1}$. Only
if $r/r'$ is the square of an element of this subgroup the sum in
\eqref{dt-def} is non-empty, and then consists of two equal summands.

The delta term is
\begin{equation}\label{Dt-def}
\Dt^{\k,r;\k',r'}(\ph) \= \frac{2\vol(\RR^d/M_\k)}{(2\pi)^d} \,
\dt_{\k,\k'}\, \dt_\k(r,r') \, \npl(\ph)\,,
\end{equation}
where $\npl=\otimes_j \npl_{\x_j}$ is the Plancherel measure in
\eqref{pl-def} written in terms of the spectral parameter~$\nu$:
\begin{align}
\int f\, d\npl_0 &\= 2i \int_0^{i\infty} f(\nu) \,\tan\pi \nu\; \nu\,
d\nu \\
\nonumber
&\qquad \hbox{} + \sum_{b\geq 2\,,\; b\equiv 0\bmod 2} (b-1)\; f\left(
\txtfrac{b-1}2\right)\,,
\displaybreak[0]\\
\nonumber
\int f\, d\npl_1 &\= -2i \int_0^{i\infty} f(\nu)\,
\cot\pi\nu\;\nu\,d\nu\\
\nonumber
&\qquad \hbox{} + \sum_{b\geq 3\,,\; b\equiv 1\bmod 2} (b-1) \;
f\left( \txtfrac{b-1}2\right)\,.
\end{align}

\subsubsection{Sum of Kloosterman sums}The Bessel transform
$\B^\sg_\x$ is the same as in (34) of \cite{densII}, with $\sg\in
\{1,-1\}^d$, $\sg_j = \sign(r_j)$. For each test function $\ph$ it
provides a function $\B^\sg \ph$ on $(\RR^\ast)^d$. The Kloosterman
term in the formula is
\begin{equation}\label{K-def}
\K^{\k,r;\k',r'}(\B^\sg \ph) \= \sum_{c\in \hg{\k'}\Ccal^\k} \frac{
\kls\ch(\k',r';\k,r;c)}{|\nrm c|} \; \B^\sg \ph\left(
\txtfrac{4\pi\sqrt{rr'}}c\right)\,.
\end{equation}

We restrict ourselves to stating the formula in the equal sign case
$\sign(r)=\sign(r')$, since this is the case needed in this paper.
With a different Bessel transform, the formula goes through if
$\sign(r)\neq\sign(r')$.

\begin{thm}\label{thm-sf} {\rm (Spectral sum formula) }Let $\k,\k'\in
\Pcal$ and $r\in \tilde M_\k'\setminus \{0\}$, $r'\in \tilde
M_{\k'}'\setminus\{0\}$ such that $\sg=\sign(r) = \sign
(r')$. For any test function $\ph$ the sums and integrals
$\nctr(\ph)$, $\ectr(\ph)$, $\npl(\ph)$ and
$\K^{\k,r;\k',r'}(\B^\sg \ph)$ converge absolutely, and
\begin{align*} &\frac{\nctr(\ph) + \ectr(\ph)}{\vol(\G\backslash G)}
\\
&\qquad \hbox{} \= \frac{2\vol(\RR^d/M_\k)}{(2\pi)^d} \,\dt_{\k,\k'}
\,\dt_\k(r,r')\, \npl(\ph) + \K^{\k,r;\k',r'}(\B^\sg\ph)\,.
\end{align*}
\end{thm}

\subsection{Proof of the sum formula} We shall go through the proof in
\S3 of \cite{BM9}, indicating where changes are needed to deal with
Fourier coefficients at cusps $\k\neq\infty$.
\subsubsection{Poincar\'e series} Let $\k \in \Pcal$. If the function
$h^\k$ on $G$ satisfies the transformation rule
\[ h^\k(g_\k n(x)g) \= e^{2\pi i S(rx) } h^\k(g_\k g)\]
with $r\in \tilde M'_\k$ and the estimate $h^\k( g_\k n a(y) k ) \ll
\prod_j \min\left( y_j^\al, y_j^{-\bt}\right)$ with $\al>1$ and
$\al+\bt>0$, then
\begin{equation}\label{Pdefk}
P^\k h^\k(g) = \sum_{\g\in \GmN\k \backslash\G} \ch(\g)^{-1} h^\k(\g
g)
\end{equation}
converges absolutely and defines the function $P^\k h^\k$ on $G$ that
is a square integrable, $(\G,\ch)$-automorphic function. In the first
step of the proof of convergence, the sum over the units in Lemma 2.3
of \cite{BM9} is replaced by a sum over $\GmN\k\backslash\GmP\k$. The
method in \S8 of \cite{BM98} works well for this sum. The second step
is a reduction to the convergence of the Eisenstein series, which is
a fact that we can assume. In Lemma 2.4 of \cite{BM9}, in place of
(2.52) and (2.53), we have, as $N(y):=\prod_j y_j\rightarrow\infty$:
\begin{align}
P^\k h^\k(g_\k n a(y)k) &\;\ll_{\al,\bt,\e}\; \max\bigl(
N(y)^{1-\al+\e}, N(y)^{-\bt+\e}\bigr)\,,\displaybreak[0]\\
\nonumber
P^\k h^\k(g_\ld n a(y)k) &\;\ll_{\al,\e}\; N(y)^{1-\al+\e}\,,
\end{align}
for $\ld\in \Pcal\setminus\{\k\}$.

For $h^\k$ and $f$ of weight~$q$, equation
(3.1) in \cite{BM9} takes the form
\begin{equation}
\langle P^\k h^\k,f\rangle\= \frac{\vol( \RR^d/M_\k
)}{\vol(\G\backslash G)} \int_A h^\k(g_\k a) \overline{F_{\k,r}
f(a)}\, |a|^{-1}\, da\,.
\end{equation}

\subsubsection{Fourier coefficient of Poincar\'e series} For $(\k,r)$
and $(\k',r')$ as in Theorem~\ref{thm-sf} and for $h^{\k'}$
satisfying the conditions above:
\begin{align}
F_{\k,r} P^{\k'} h^{\k'}(g) &\= \frac1{\vol(\RR^d/M_\k)} \sum_{\g\in
\GmN{\k'}\backslash \G} \ch(\g)^{-1}\\
\nonumber
&\qquad \cdot \int _{\RR^d/ M_\k} e^{-2\pi iS(rx)} h^{\k'}(\g g_\k
n(x) g)\, dx\,.
\end{align}
We write $g_{\k'}^{-1}\g g_\k=\matc abcd$. Only if $\k=\k'$ there can
be a contribution with $c=0$. With $\g=g_\k \matc \e\bt0{1/\e}
g_\k^{-1}$ as in \S\ref{sss-dt}, we obtain:
\begin{align}\label{fdt}
&\frac1{\vol(\RR^d/M_\k)} \sum_{\g\in \GmN\k \backslash\GmP\k}
\ch(\g)^{-1}\\
\nonumber
&\qquad\qquad\hbox{} \cdot
 \int_{\RR^d/M_\k} e^{-2\pi i S(rx)} h^\k\left( g_\k n(\bt\e+\e^2 x)
h(\e) g\right)\, dx\displaybreak[0]\\
\nonumber
&\quad\= \frac1{\vol(\RR^d/M_\k)} \int_{\RR^d/M_\k} e^{2\pi i
S((r-\e^2r')x)} \, dx \; h^\k\left( g_\k h(\e) g\right)
\displaybreak[0]\\
\nonumber
&\quad\= 2 \overline{\dt_\k(r,r')} \, h^\k(g_\k a(\e^2) g)\,,
\end{align}
where $\e\in \Ocal_F^\ast$ satisfies $\e^2=r/r'$ if there are elements
of the form $\matc \e\bt0{1/\e}$ in $ g_\k^{-1}\G g_\k$.

For all combinations of $\k$ and~$\k'$ the number of terms with
$c\neq0$ is large. We write $g_{\k'}^{-1}\g g_\k = \matc abcd =
n(a/c)
w(c)n(d/c)$, where $w(y)=\matr0{-y^{-1}}y0$, and obtain, in the
notations in~\S\ref{sect-not}:
\begin{align}\label{psk}
&\frac1{\vol(\RR^d/M_\k)} \sum_{c\in \hg{\k'}\Ccal^\k} \sum_{\g\in
\hg{\k'}\Scal^\k(c)} \sum_{\dt\in \GmN\k} \ch(\g\dt)^{-1}\\
\nonumber
&\qquad\qquad\hbox{} \cdot
\int_{\RR^d/M_\k} e^{-2\pi i S(rx)} h^{\k'}\left( \g g_\k n(x)g
\right)\, dx\displaybreak[0]\\
\nonumber
&\quad\= \frac1{\vol(\RR^d/M_\k)}\sum_{c\in \hg{\k'}\Ccal^\k}
\sum_{\g\in \hg{\k'}\Scal^\k(c)} \ch(\g)^{-1}\\
\nonumber
&\qquad\qquad\hbox{} \cdot
 \int_{\RR^d} e^{-2\pi i S(r(x-d/c))+2\pi i S(r'a/c)} h^{\k'}\left(
 g_{\k'} w(c)
n(x)g\right)\, dx\displaybreak[0]\\
\nonumber
&\quad\= \sum_{c\in \hg{\k'} \Ccal^\k}
\frac{\kls\ch(\k,r;\k',r';c)}{\vol(\RR^d/M_\k)} \int_{\RR^d} e^{-2\pi
i S(rx)} h^{\k'}\left(g_{\k'} w(c) n(x)g\right)\, dx\,.
\end{align}

\subsubsection{Whittaker transform}{\it \S3.2} goes through almost
verbatim. (We again indicate references to \cite{BM9} by italics.)
In {\it Definition 3.4}, we put
\begin{equation}
w^{\k,r}_q \eta( g_\k g) \= w_q^r \eta(g)\,.
\end{equation}
We insert $g_\k$ at appropriate places. The last formula in {\it
Theorem~3.8}, implies
\begin{equation}
\int_{\RR^d} e^{-2\pi i S(rx)} w^{\k',r'}_q\eta\left( g_{\k'} w(c)
n(x)g\right)\, dx \= w^{\k,r}_q\tilde \eta(g_\k g)\,.
\end{equation}

\subsubsection{Restricted version of the formula}As in {\it \S3.3},
the scalar product of Poincar\'e series $P^\k w^{\k,r}_q \eta$ and
$P^{\k'}w^{\k',r'}_q \eta'$ is computed in two ways.

In the spectral computation, we obtain the following
\begin{align*} \langle P^\k w^{\k,r}_q \eta, \ps_{\varpi,q}\rangle &\=
8^{d/2}\,\pi^d \, \Bigl( \frac{\vol(\RR^d/M_\k)}{\vol(\G\backslash
G)} \Bigr)^{1/2} \,|\nrm r|^{1/2}\, \overline{c^{\k,r}(\varpi)}\\
&\qquad\hbox{} \cdot
\eta(\nu_\varpi)\, \prod_j \frac{e^{-\pi i
q_j}}{\Gf(\frac12+\bar\nu_{\varpi,j}
+ \frac{q_j\sg_j}2)}\,,
\end{align*}
together with a similar expression for the scalar product with an
Eisenstein series.
(Compare with {\it(3.39)}.)
We have to modify {\it(3.40)} and {\it(3.41)} by inserting the cusps
$\k$ and $\k'$ into the Fourier coefficients. We modify the
definition of the measure by defining $d\s^{\k,r;\k',r'}_{\ch,\x}$ by
the expression in {\it(3.43)} with $\k$ and $\k'$ inserted into the
Fourier coefficients. Then we obtain in place of~{\it(3.46)}:
\begin{equation}
\langle P^\k w^{\k,r}_q\eta, P^{\k'}w^{\k',r'}_q\eta'\rangle \=
\frac{(8\pi^2)^d\, \sqrt{\nrm{rr'}}}{\vol(\G\backslash G)}\,
\int_{Y_\x} \th^{r,r'}_q(\nu)\, d\s^{\k,r;\k',r'}_{\ch,\x}(\nu)\,,
\end{equation}
with
\[ \th^{r,r'}_q(\nu) = \prod_{j=1}^d \frac{ \eta_j(\nu_j)
\overline{\eta'_j(\bar\nu_j})}{\Gf(\frac12-\nu_j+\frac{q_j\sign r_j}2)
\Gf(\frac12-\nu_j+\frac{q_j\sign r_j}2)}\,,\]
and $d\s^{\k,r;\k',r'}_{\ch,\x}$ adapted to the new normalization.

A comparison shows that the cuspidal subspace contributes
\[(8\pi^2)^d\, \sqrt{\nrm{rr'}}\, \vol(\G\backslash G)^{-1}\,
\nctr(\th^{r,r'}_q)\,,\]
with a similarly modified expression for the scalar product with
Eisenstein series.

The geometric computation of the scalar product is carried out as in
{\it\S3.3.4}. The contribution from \eqref{fdt} is equal to
\begin{equation}
2\dt_{\k,\k'} \,\dt_\k(r,r')\, \vol(\RR^d/M_\k)\,
(4\pi)^d |\nrm r| \int_{Y_\x} \th^{r,r'}_q(\nu) \, d\npl(\nu)\,.
\end{equation}
This agrees with {\it(3.50)}, with the substitutions
\begin{equation}
\al(\ch,\x;r,r') \mapsto 2 \dt_{\k,\k'}\dt_\k(r,r')\,,\quad
\sqrt{|D_F|} \mapsto \vol(\RR^d/M_\k)\,.
\end{equation}

The remaining contribution to the scalar product is given
by~\eqref{psk}:
\begin{align}\label{n}
\vol(\RR^d/M_\k) & \int_A w^{\k,r}_q\eta(g_\k a) \cdot \sum_{c\in
\hg{\k'}\Ccal^\k} \vol(\RR^d/M_\k)^{-1}
\overline{\kls\ch(\k,r;\k',r')}\\
\nonumber
&\qquad\hbox{} \cdot
\int_{\RR^d} e^{2\pi i S(rx)} \overline{ w^{\k',r'}_q \eta'\left(
g_{\k'} w(c)
n(x)a\right)}\, dx\; |a|^{-1}\, da \,.
\end{align}
Using \eqref{Sconj} and {\it Theorem~3.8}, we see that the expression
in \eqref{n} is equal to
\[ \sum_{c\in \hg{\k'}\Scal^\k} \kls\ch(\k',r';\k,r) \, \ch(-1) \int_A
w^{\k,r}_q\eta(g_\k a) \, \overline{w^{\k,r}_q\tilde\eta'(g_\k a)}\,
|a|^{-1}\, da\,.\]
The transition $r\mapsto r'$ under the conjugation is present in {\it
Theorem~3.8}. The transition $\k'\mapsto \k$ is a consequence of the
definitions, and has to be checked. With {\it(3.52)} the integral
over~$A$ is given by
\[ \frac{(8\pi^2)^d}{|\nrm c|} \, \nrm{rr'}^{1/2}\, \ch(-1)\, (\B^\sg
\th_q^{r,r'})
\bigl(\txtfrac{4\pi\sqrt{rr'}}c\bigr)\,.\]
This gives the final result of this contribution:
\begin{equation}
(8\pi^2)^d \, \nrm{rr'}^{1/2}\, \K^{\k,r;\k',r'}\bigl( \B^\sg
\th^{r,r'}_q\bigr)\,.
\end{equation}

Division by $(8\pi^2)^d \,\nrm{rr'}^{1/2}$ gives the restricted sum
formula
\begin{align} \label{rsf}
\frac1{\vol(\G\backslash G)}\, \int_{Y_\x} \th(\nu)\,
d\s^{\k,r;\k',r'}_{\ch,\x}(\nu)
&\= \dt_{\k,\k'} \frac{2\vol(\RR^r/M_\k)}{(2\pi)^d} \int_{Y_\x}
\th(\nu)\, d\npl(\nu)\\
\nonumber
&\qquad\hbox{} + \K^{\k,r;\k',r'}(\B^\sg\th)\,.
\end{align}
for all $\th$ indicated in {\it Proposition~3.9}.

At this point we notice a minor error in \cite{BM9}. With the sum of
Kloosterman sums as defined in {\it(3.60)}, it should read
$\K^{r',r}_\ch$ instead of $K^{r,r'}_\ch$ in {\it(3.61)}. The
solution that we will adopt from now on, is 
to take the expression in {\it(3.60)} as the definition of
$\K^{r,r'}_\ch$, like in \eqref{K-def} here.

The proof of the convergence of the sum of Kloosterman sums
{\it(Proposition 3.14)} has to be revisited. In \cite{BM9} the prime
ideals are split according to whether they divide the ideal~$I$ or
not. The finite set $S$ in Proposition~\ref{prop-Kl-est} may tun out
to be a larger set of prime ideals than those dividing~$I$.
Nevertheless, the method in {|it(3.72)}, goes through.

\subsubsection{Extension}The rest of the proof of the formula in
\cite{BM9} is based on the restricted formula in {\it
Proposition~3.9}, and consists of extending the space of test
functions for which the formula holds. The dependence on the cusps
$\k$ and $\k'$ in \eqref{rsf} here, is immaterial in these extension
steps. In {\it\S3.5} we use {\it Proposition 3.16}, which goes
through unchanged. Thus we arrive at the sum formula as stated in
Theorem~\ref{thm-sf}.

\section{Asymptotic formula} \label{sect-af} In the proof of
Proposition~\ref{prop-asfp} we have used the generalization
\eqref{asf} of \eqref{as}, where
$|c^r(\varpi)|^2 = |c^{\infty,r}(\varpi)|^2$ is replaced by
$\overline{ c^{\k,r}(\varpi)}\, c^{\k',r'}(\varpi)$. The aim of this
section is to show that the methods in~\cite{densII} can be extended
to give the asymptotic result~\eqref{asf} below.\medskip

We define for $\k,\k'\in \Pcal$, $r,r'\in \tilde M_\k'\setminus 0$ and
compact sets $\Om \subset \RR^d$
\begin{equation}\label{N'}
\ctr (\Om) \= \sum_{\varpi\,,\; \ld_\varpi\in \Om}
\overline{c^{\k,r}(\varpi)}\, c^{\k',r'}(\varpi)\,,
\end{equation}
where $\varpi$ runs through a maximal orthogonal system of irreducible
subspaces of $L^{2,\mathrm{cusp}}_\x(\G\backslash G;\ch)$. This
generalizes the counting function on the left hand side
of~\eqref{as1}.
\begin{thm}\label{thm-af'}Let $t\mapsto \Om_t$ be a family of bounded
sets in~$\RR^d$ as in \eqref{Om-gen}, or satisfying the conditions
indicated in~\S\ref{sect-cond} below. Let $\k,\k'\in \Pcal$, and let
$r\in \tilde M_\k'\setminus\{0\}$, $r'\in\tilde
M_{\k'}'\setminus\{0\}$, such that $\sign r = \sign r'$. Then, as
$t\rightarrow\infty$
\begin{equation}\label{asf}
\ctr(\Om_t) \= \dt_{\k,\k'} \, \dt_\k(r,r')
\, \frac{2\vol(\RR^d/M_\k)\,\vol(\G\backslash G)}{(2\pi)^d} \,
\pl(\Om_t) + o\bigl( V\lw1(\Om_t)\bigr)\,,
\end{equation}
with the notations in \eqref{Dt-def}, \eqref{pl-def}
and~\eqref{V-def}.
\end{thm}

The restriction to $\sign r=\sign r'$ is necessary to us. We have not
been able to estimate Bessel transforms suitably in the unequal sign
case.

\subsection{Conditions}\label{sect-cond}In \S\ref{sect-das} we will
show that Theorem~\ref{thm-af'} is valid for families
$t\mapsto \Om_t$ as used in~\cite{densII}. Such families have product
form
\[ \Om_t \= \hat C_t^+ \times \hat C_t^- \times \prod_{j\in E}
[A_j,B_j] \,,\]
based on a partition $\{1,\ldots,d\}= Q^+ \sqcup Q^- \sqcup E$ of the
archimedean places of~$F$. The bounded intervals $[A_j,B_j]$ do not
depend on~$t$, and the endpoints should not be of the form $\frac
b2\bigl(1-\nobreak\frac b2\bigr)$ with $b\equiv \x_j \bmod 2$, $b>1$.
The sets $\hat C_t^+$ and $\hat C_t^-$ are compact sets contained in
$\prod_{j\in Q^+} \bigl[\frac54,\infty\bigr)$, respectively
$\prod_{j\in Q^-}
(\infty,0]$, such that the corresponding sets $C_t^\pm = \prod_{j\in
Q^\pm} \bigl( i[1,\infty)
\cup [0,\infty)
\bigr)$ in the variable $\nu$ with $\ld=\frac14-\nu^2$ satisfy the
conditions in (97) or (102) in~\cite{densII}. The proof of
Theorem~\ref{thm-afH} is based on the asymptotic formula, so the
statement of the theorem holds for all families $t\mapsto \Om_t$
satisfying these conditions.

A family $t\mapsto \Om_t$ as in~\eqref{Om-gen} in the introduction
does not satisfy these conditions directly. We write
$\Om_t = \bigsqcup_p \Om^{(p)}_t$ with
\[ \Om^{(p)}_t \= \bigl[\txtfrac 54,t\bigr]^{Q^+} \times
\bigl[-t,-\txtfrac12\bigr] \times \bigl[-\txtfrac12,\txtfrac
54\bigr]\,,\]
and let $p=(Q^+,Q^-,Q^0)$ run over the partitions $Q=Q^+\sqcup Q^-
\sqcup Q^0$ such that $Q^0\neq Q$. All sets $\Om_t^{(p)}$ satisfy the
conditions in~\cite{densII}. Summing the asymptotic formulas
\eqref{asf} applied to each of the families $t\mapsto \Om_t^{(p)}$
gives it for~$t\mapsto \Om_t$.

There are more families for which the asymptotic formula can be
proved, by expressing them in families satisfying the conditions in
(97) or (102)
in~\cite{densII}. Hence the statement of Theorem~\ref{thm-afH} holds
for all these families, in particular for the families in \S1.2.4--13
and \S6 in~\cite{densII}.

\subsection{Estimates of Fourier coefficients of Eisenstein series}In
\cite{densII}, (32),
(33), we quoted from \cite{BMP3a} and~\cite{BM9} estimates of
 $D^{\infty,r}_\x(\ld,\ch;iy,i\mu)$. These estimates have to be
 generalized to $D^{\k,r}_\x(\ld,\ch;iy,i\mu)$.

The estimations in \S5.1--2 of \cite{BMP3a} are based on the fact that
Eisenstein series for $\G_0(I)$ are linear combinations of Eisenstein
series for the subgroup $\G(I)$. In \S4.2 of \cite{BM9} it is shown
that this result extends to the situation with weights in $\ZZ$
instead of $2\ZZ$. The character $\ch$ is trivial on $\G(I)$, hence
it only influences the coefficients in the linear combination.

These estimations concern the Fourier coefficients $D^{\infty,r}_\x$.
Let $\k\in \Pcal$ be another cusp. The function $\tilde E_q:g \mapsto
E_q(\ld,\ch;\nu,i\mu;g_\k g)$ is an Eisenstein series on the group
$\G_1=g_\k^{-1} \G_0(I) g_\k$ for a character $\ch_1$ determined by
conjugation. Actually, depending on the normalizations, it might be a
multiple of an Eisenstein series on~$\G_1$, with a factor in which
the influence of~$\nu$ and $\mu$ is of the form $t^{\nu} t_1^{i\mu}$
with $t>0$, $t_1>0$. So this factor is unimportant for estimates.
Since $g_\k\in \SL_2(F)$, the group $\G_1$ is commensurable with
$\G_0(I)$, and contains a principal congruence subgroup $\G(I_1)$,
with $I_1 \subset I$. The character $\ch$ is trivial on $\G(I)$, so
we can arrange $I_1$ such that $\ch_1$ is trivial on $\G(I_1)$. Thus,
for the Fourier coefficients of $\tilde E_q$ at $\infty$ with nonzero
order, we have an estimate like in Proposition~4.2 in~\cite{BM9}. The
Fourier coefficients of $E_q(\ld,\ch;\nu,i\mu)$ at $\k$ of nonzero
order can be expressed in those of~$\tilde E_q$ at~$\infty$. The
consequence is that $D^{\k,r}_\x(\ld,\ch;\nu,i\mu)$ satisfies an
estimate like that in Proposition~4.2 of~\cite{BM9}. As in (33) of
\cite{densII}, we have:
\begin{equation}\label{Eis-est}
D^{\k,r}_\x(\ld,\ch;iy,i\mu)
\ll_{F,I,\k,r,\k',r'}\biggl( \log\bigl(2+\sum_j |y+\mu_j| \bigr)
\biggr)^7\,.
\end{equation}

\subsection{Derivation of the asymptotic formula}\label{sect-das}
The method of proof is the same as in \S2-5 in \cite{densII}, so we
shall only indicate the points where the argument departs from the
approach therein.

In the proofs in~\cite{densII} we mostly use the spectral parameter
$\nu$ instead of the eigenvalue $\ld=\frac14-\nu^2$. In terms of the
spectral parameter the counting function in \eqref{N'} has the form
\begin{equation}
\nctr (\tilde\Om) \;=\; \sum_{\varpi\,,\; \nu_\varpi\in \tilde\Om}
\overline{c^{\k,r}(\varpi)}\, c^{\k',r'}(\varpi)\,.
\end{equation}

We choose the test functions in the same way as in Lemma 2.2
in~\cite{densII}. The considerations in \S2.2.1 and \S2.2.4, {\sl
loc. cit.}, do not depend on the Fourier term order, and go through
unchanged. We use the Weil type estimate of Kloosterman sums in
Proposition~\ref{prop-Kl-est}, and, in \S2.2.2, {\sl loc. cit.}, we
replace the distinction between $\prm\divides I$ and $\prm\nmid I$,
by the distinction, $\prm\in S$ and $\prm\not\in S$, where $S$ is a
finite set of primes as in Proposition~\ref{prop-Kl-est}.

In \S2.2.3 of \cite{densII} we use the estimate \eqref{Eis-est} of
Fourier coefficients of Eisenstein series. The statement of
Proposition 2.4 in \cite{densII} goes through for the counting
quantity $\tilde N^{\k,r;\k',r'}(\ph(q,\cdot))$. The factor
$\sqrt{|D_F|}$ is a specialization of $\vol(\RR^d/M_\k)$, and a
factor $\dt_{\k,\k'} \, \dt_\k(r,r')$ is inserted. Hence the delta
term is present only if the cusps $\k$ and $\k'$ are equal, and if
also $r/r'$ is the square of a generalized unit.

The remaining proofs in \cite{densII} are based on the estimate in
Proposition~2.4 of~\cite{densII}, and hence go through for Fourier
coefficients at different cusps.


\newcommand\bibit[4]{
\bibitem {#1}#2: {\em #3;\/ } #4}


\end{document}